\documentclass[11pt,oneside]{amsart}

\usepackage{amsfonts}
\usepackage{amsmath, amsthm, amssymb, amscd}
\usepackage[mathscr]{euscript}

\def\crn#1#2{{\vcenter{\vbox{
        \hbox{\kern#2pt \vrule width.#2pt height#1pt
           }
          \hrule height.#2pt}}}}

\newcommand{\stopthm}{\hfill$\square$\medskip}

\newcommand{\pa}{\partial}
\newcommand{\Ric}{\operatorname{Ric}}

\newcommand{\Vol}{\operatorname{Vol}}

\newcommand{\tr}{\operatorname{tr}}

\newcommand{\tf}{\operatorname{tf}}

\newcommand{\D}{\Delta}
\newcommand{\R}{\mathbb R}

\newcommand{\ep}{\epsilon}
\newcommand{\fe}{\varphi}
\newcommand{\al}{\alpha}

\newcommand{\ga}{\gamma}
\newcommand{\de}{\delta}
\newcommand{\si}{\sigma}
\newcommand{\be}{\beta}

\newcommand{\om}{\omega}
\newcommand{\Up}{\Upsilon}

\newcommand{\gbh}{\widehat{\overline{g}}}
\newcommand{\gb}{\overline{g}}

\newcommand{\nub}{\bar \nu}
\newcommand{\nab}{\overline{\nabla}}

\newcommand{\Jb}{\overline{J}}
\newcommand{\Lo}{\mathring{L}}
\newcommand{\Rb}{\overline{R}}
\newcommand{\Wb}{\overline{W}}
\newcommand{\Pb}{\overline{P}}
\newcommand{\Db}{\overline{\Delta}}
\newcommand{\Gb}{\overline{\Gamma}}

\newcommand{\Vt}{\widetilde{V}}

\newcommand{\Et}{\widetilde{E}}

\newcommand{\Om}{\Omega}
\newcommand{\Si}{\Sigma}

\newcommand{\cL}{\mathcal{L}}

\newcommand{\cF}{\mathcal{F}}
\newcommand{\cE}{\mathcal{E}}

\newcommand{\cI}{\mathcal I}

\newcommand{\cP}{\mathcal{P}}
\newcommand{\cB}{\mathcal{B}}

\theoremstyle{plain}
\newtheorem{theorem}{Theorem}[section]
\newtheorem{lemma}[theorem]{Lemma}
\newtheorem{proposition}[theorem]{Proposition}
\newtheorem{corollary}[theorem]{Corollary}

\theoremstyle{definition}

\theoremstyle{remark}
\newtheorem{remark}[theorem]{Remark}

\numberwithin{equation}{section}

\title[Chern-Gauss-Bonnet Formula]{Chern-Gauss-Bonnet Formula for Singular
  Yamabe Metrics in Dimension Four}

\author{C. Robin Graham}
\address{Department of Mathematics, University of Washington,
Box 354350\\
Seattle, WA 98195-4350}
\email{robin@math.washington.edu}

\author{Matthew J. Gursky}
\address{Department of Mathematics, University of Notre Dame,\\
Notre Dame, IN 46556}
\email{mgursky@nd.edu}

\begin{document}

\begin{abstract}
We derive a formula of Chern-Gauss-Bonnet type for the Euler characteristic
of a four dimensional manifold-with-boundary in terms of the geometry of
the Loewner-Nirenberg singular Yamabe metric in a prescribed conformal
class.  The formula involves the renormalized volume and a boundary
integral.  It is shown that if the boundary is 
umbilic, then the sum of the renormalized volume and the boundary integral
is a conformal invariant.  Analogous results are proved for asymptotically  
hyperbolic metrics in dimension four for which the second elementary  
symmetric function of the eigenvalues of the Schouten tensor is constant.
Extensions and generalizations of these results are discussed.  Finally, a 
general result is 
proved identifying the infinitesimal anomaly of the renormalized volume of
an asymptotically hyperbolic metric in terms of its renormalized volume 
coefficients, and used to outline alternate proofs of the conformal
invariance of the renormalized volume plus boundary integral.   

\end{abstract}

\maketitle

\thispagestyle{empty}

\renewcommand{\thefootnote}{}
\footnotetext{The second author acknowledges the support of NSF grants
  DMS-1811034 and DMS-1547292.}

\renewcommand{\thefootnote}{1}

\section{Introduction}\label{introduction}

In this paper we derive a Chern-Gauss-Bonnet formula for singular Yamabe
metrics in dimension 4, and also analyze related questions for metrics  
solving certain generalizations of the singular Yamabe problem.  In our 
context, a singular 
Yamabe metric means an asymptotically hyperbolic metric $g$ on the interior
of a smooth, compact, connected manifold-with-boundary $(M^{n+1},\pa M)$
with constant scalar curvature $R=-n(n+1)$.  Contained in the class of
singular Yamabe metrics are the Poincar\'e-Einstein metrics: those
asymptotically hyperbolic metrics satisfying $\Ric(g)=-ng$.  Clearly any
Poincar\'e-Einstein metric is a singular Yamabe metric (ignoring here the
issue of the boundary regularity which is assumed).

The basic result concerning singular Yamabe metrics
(\cite{LN}), \cite{AM}, \cite{M}, \cite{ACF}) is that for $n\geq 2$, given
any smooth metric $\gb$ on $M$, there exists a unique defining function $u$
for $\pa M$ so
that $g=u^{-2}\gb$ is a singular Yamabe metric.  The metric $g$ depends
only on the conformal class $[\gb]$ determined by $\gb$:  if $\gb$ is
replaced by $\Om^2\gb$ with $0<\Om\in C^\infty(M)$, then $u$ is replaced by
$\Om u$ so that $g$ is unchanged.  If $\rho$ is any defining function
for $\pa M$ (not necessarily $C^\infty$), the metric $\gb=\rho^2g$ is
called a compactification of $g$.  A smooth metric $\gb$ is in particular a
compactification of the singular Yamabe metric which it determines,
realized by taking $\rho = u$.  We will denote by $h=\gb|_{T\pa M}$ the
metric on $\pa M$ induced by $\gb$.

It follows from \cite{M}, \cite{ACF} that the defining function $u$
determined by a smooth metric $\gb$ has an asymptotic expansion of the form 
\begin{equation}\label{uexpand}
u=r+u^{(2)}r^2+\ldots +u^{(n+1)}r^{n+1}+\cL r^{n+2}\log r +O(r^{n+2})
\end{equation}
relative to the product identification of a collar neighborhood of $\pa M$ 
induced by $\gb$ (see the beginning of \S\ref{calcs}).  In particular $r$
is the $\gb$-distance to $\pa M$.  The coefficients $\cL$ and  
the indicated $u^{(j)}$ are smooth, locally determined functions on $\pa
M$.  

Volume renormalization for singular Yamabe metrics was considered in
\cite{G3}, \cite{GoW}, generalizing the discussion for Poincar\'e-Einstein
metrics in \cite{G1}.  We follow the formulation in \cite{G3}.  
As $\ep\to 0$, 
\begin{equation}\label{sypvolexp}
\operatorname{Vol}_g(\left\{r>\ep\right\})
= c_0\ep^{-n} +c_1\ep^{-n+1}+\cdots +c_{n-1}\ep^{-1}
+\cE\log \frac{1}{\ep}+V +o(1),
\end{equation}
where each of $\cE$ and the $c_j$'s is the integral over $\pa M$ of a local
invariant of the extrinsic geometry of $\pa M$ with respect to $\gb$.  The
log coefficient $\cE$ can be viewed as an energy of the submanifold
$\pa M$ of $(M,\gb)$ which is invariant under conformal
rescalings of $\gb$.  The renormalized volume $V=V(g,\gb)$ is a globally
determined quantity which in general
depends on $\gb$.  But its anomaly, i.e. its change under conformal 
rescaling $\gbh=e^{2\om}\gb$, $\om \in C^\infty(M)$, is locally determined;
it has the form
\begin{equation}\label{anomalydef}
V(g,\gbh)-V(g,\gb)=\int_{\pa M}\cP_{\gb}(\om)\,dv_h,
\end{equation}
where $\cP_{\gb}(\om)$ is a polynomial nonlinear differential operator
determined by the local geometry of $\pa M$ in the metric $\gb$.  If $g$ is 
Poincar\'e-Einstein and one restricts to geodesic
compactifications (meaning $\gb=r^2g$ with $|dr|_{\gb}=1$ near $\pa M$),
then $c_j=0$ for $j$ odd.  If in addition $n$ is odd, then also $\cE=0$ and
$V$ is conformally invariant (see \cite{G1}).

Recall that the Chern-Gauss-Bonnet formula for a compact Riemannian
4-manifold $(M,g)$ reads
\begin{equation}\label{cgb}
8\pi^2\chi(M)=\int_M\left(\tfrac14 |W|^2 +4\si_2(g^{-1}P)\right)dv_g,
\end{equation}
where $\chi(M)$ is the Euler characteristic, $W$ is the Weyl tensor,
$|W|^2=W_{\al\be\ga\de}W^{\al\be\ga\de}$, and $\si_2(g^{-1}P)$ denotes
the second elementary symmetric function of the eigenvalues of the
endomorphism $g^{\al\ga}P_{\ga\be}$.  Here $P$ is the Schouten
tensor, given by
\[
(n-1)P=\Ric(g)-Jg, \qquad\qquad J=\tr P =\frac{R}{2n} 
\]
for a metric in dimension $n+1$.  For $n=3$, we have
\begin{equation}\label{sig2}
4\si_2(g^{-1}P)=\tfrac{1}{24}R^2-\tfrac12 |E|^2,
\end{equation}
where $E=\Ric(g)-\frac{R}{4}g$ is the Einstein tensor.  In particular, if 
$g$ has $R=-12$, then \eqref{cgb} becomes
\begin{equation}\label{cgbcsc}
8\pi^2\chi(M)=\frac14\int_M|W|^2\,dv_g -\frac12\int_M|E|^2\,dv_g +6V
\end{equation}
with $V$ the volume of $(M,g)$.  In case $g$ is also Einstein, this reduces
to
\begin{equation}\label{cgbeinstein}
8\pi^2\chi(M)=\frac14\int_M|W|^2\,dv_g +6V.
\end{equation}
In \cite{A}, Anderson showed that
\eqref{cgbeinstein} holds also if $g$ is Poincar\'e-Einstein, where now $V$
is the renormalized volume of $(M,g)$.  Conformal invariance of the
integrand shows in this case that $\int_M|W|^2\,dv_g$ is convergent.

Our formula is an analogue of \eqref{cgbcsc} for singular Yamabe metrics,
but now a boundary integral appears.  We denote by $L$ the second
fundamental form for $\pa M$ relative to $\gb$ with respect to the inward
pointing unit normal
$\nub$: $L(X,Y)=\gb(\nab_X Y,\nub)$.  $\Lo$ denotes its trace-free part,
$H=\tr_h L$ the mean curvature, and $|L|^2$ and $|\Lo|^2$ the
norms with respect to the induced metric $h$.  Curvature expressions for
$\gb$ carry an
overline (for example the scalar curvature of $\gb$ is $\Rb$), while
curvature for $g$ is unadorned (scalar curvature of $g$ is $R$).  We use
Greek indices $\al$, $\be$ for $M$ ($0\leq \al,\be\leq 3$), Latin indices
$i$, $j$ for $\pa M$ ($1\leq i,j\leq 3$), and
a $0$ index for the inward unit normal, so that a Greek index $\al$
specializes either to a $0$ or an $i$.  Thus $\Rb_{00}$ is another notation
for $\Ric_{\gb}(\nub,\nub)$, and $\Wb_{0i0j}$ denotes the section of 
$S^2T^*\pa M$ obtained by contracting the Weyl tensor for $\gb$ twice into
$\nub$ and orthogonally projecting onto $T^*\pa M$ in the other two
indices.  Define $\cB_{\gb}\in C^\infty(\pa M)$ by
\begin{equation}\label{B}
24\cB_{\gb}=\pa_{\nub}\Rb+124\Lo^{ij}\Wb_{0i0j}
+108\tr(\Lo^3)+14H|\Lo|^2 +24\Lo^{ij}\Rb_{ij}
+6 H\Rb_{00} -\tfrac{10}{3} H\Rb  -\tfrac{16}{9} H^3.
\end{equation}

\begin{theorem}\label{main}
Let $g$ be a singular Yamabe metric in dimension 4 and
$\gb$ a smooth compactification of $g$.  Then
\begin{equation}\label{mainformula}
8\pi^2\chi(M)  = \frac14 \int_M|W|^2\,dv_g
- \frac12 \text{ fp} \int_{r>\ep}|E|^2\,dv_g +6V(g,\gb)
+\int_{\pa M}\cB_{\gb} \,dv_h.
\end{equation}
\end{theorem}
The integral $\int_M|E|^2\,dv_g$ typically diverges.  As above, $r$ denotes
the $\gb$-distance to $\pa M$, and as will explained in more detail in
\S\ref{calcs}, fp denotes the finite part of the integral:  this is
the constant term in the expansion of $\int_{r>\ep}|E|^2\,dv_g$ in
powers of $\ep^{-1}$ and $\log\ep$.  Typically each of the last three terms
on the right-hand side of \eqref{mainformula} depends on the choice of
compactification $\gb$.

Recall that a hypersurface is said to be umbilic with respect to a
background metric if $\Lo=0$.  This condition is invariant under conformal
rescalings
of the background metric.  We will say that $\pa M$ is umbilic for a
singular Yamabe metric $g$ if it is umbilic for any compactification $\gb$.
This is an important special case which includes all Poincar\'e-Einstein
metrics.
If $\pa M$ is umbilic, all the terms involving $\Lo$ drop out in
\eqref{B}, which therefore simplifies to
\begin{equation}\label{Bumbilic}
24\cB_{\gb}=\pa_{\nub}\Rb+6 H\Rb_{00} -\tfrac{10}{3} H\Rb-\tfrac{16}{9}
H^3.
\end{equation}
Recall from \cite{Es} that any Riemannian metric on a 4-dimensional 
manifold-with-boundary can be conformally rescaled to a Yamabe metric
having constant scalar curvature and for which $\pa M$ is minimal,
i.e. $H=0$.  Observe
that \eqref{Bumbilic} implies that $\cB_{\gb}=0$ if $\pa M$ is umbilic and
$\gb$ is chosen to be such a Yamabe representative in the conformal class.
It holds also that $\cB_{\gb}=0$ in case $\gb$ is a geodesic
compactification of a Poincar\'e-Einstein metric.  Then $H=0$ and
$\pa_{\nub}\Rb=0$, for instance by parity considerations.

We will see in \S\ref{calcs} that the following proposition follows from an
easy calculation
of the leading asymptotic term in the Einstein tensor.
\begin{proposition}\label{umbilicE}
Suppose $n\geq 2$.  If $g$ is a singular Yamabe metric with $\pa M$
umbilic, then $|E|_{\gb}\in L^\infty(M)$.
\end{proposition}

\noindent
Since $|E|^2_g\,dv_g = |E|^2_{\gb}\,dv_{\gb}$ when $n=3$,
Proposition~\ref{umbilicE} implies that $\int_M|E|^2_g\,dv_g<\infty$ if
$\pa M$ is umbilic and $n=3$.  So in this case $\text{ fp}
\int_{r>\ep}|E|^2\,dv_g = \int_M|E|^2\,dv_g$.  In particular,
$\text{ fp} \int_{r>\ep}|E|^2\,dv_g$ is independent of choice of $\gb$ in
the $n=3$ umbilic case.  In this case, Theorem~\ref{main} therefore
becomes:
\begin{theorem}\label{mainumbilic}
Let $g$ be a singular Yamabe metric in dimension 4 with $\pa M$ umbilic and
let $\gb$ be a smooth compactification of $g$.  Then
\begin{equation}\label{mainformulaumbilic}
8\pi^2\chi(M)  = \frac14 \int_M|W|^2\,dv_g
- \frac12 \int_M|E|^2\,dv_g +6V(g,\gb)
+\int_{\pa M}\cB_{\gb} \,dv_h.  
\end{equation}
\end{theorem}

Let $g$ be a singular Yamabe metric with $n=3$ and $\pa M$ umbilic and let
$\gb$ be a compactification of $g$.  Set
\[
\Vt(g) = V(g,\gb)+\frac16 \int_{\pa M}\cB_{\gb} \,dv_h.
\]
The notation is justified by:
\begin{corollary}\label{Vtinvariance}
If $\pa M$ is umbilic, then $\Vt(g)$ is conformally invariant, i.e. it
is independent of the choice of compactification $\gb$.
\end{corollary}

\noindent
Corollary~\ref{Vtinvariance} is an immediate consequence of
Theorem~\ref{mainumbilic} since $\chi(M)$,
$\int_M|W|^2\,dv_g$, and $\int_M|E|^2\,dv_g$ are all independent of the
choice of compactification.  Note that \eqref{mainformulaumbilic} can be
written  
\begin{equation}\label{rewrite}
8\pi^2\chi(M)  = \frac14 \int_M|W|^2\,dv_g
- \frac12 \int_M|E|^2\,dv_g +6\Vt(g),
\end{equation}
in which each term on the right-hand side is an invariant of $g$, i.e. is
independent of choice of $\gb$.  By the observation above,
$\Vt(g)=V(g,\gb)$ if the representative $\gb$ is chosen to be a Yamabe
representative having constant scalar curvature and $H=0$.  Thus choosing
such a Yamabe representative can be regarded as a sort of ``conformal gauge 
fixing'' for $\Vt(g)$.    

In the general, not necessarily umbilic, case, Theorem~\ref{main} implies
instead that  
\begin{equation}\label{combination}
6V(g,\gb) - \frac12 \text{ fp} \int_{r>\ep}|E|^2\,dv_g
+\int_{\pa M}\cB_{\gb} \,dv_h 
\end{equation}
is conformally invariant.  A direct proof of this is outlined at the end of
\S\ref{anomalies}.  

Equation~\eqref{rewrite} has the following consequence.
\begin{proposition}\label{PEineq}
Let $g$ be a singular Yamabe metric with $\pa M$ umbilic.  Then
\begin{equation}\label{ineq}
8\pi^2\chi(M)  \leq \frac14 \int_M|W|^2\,dv_g +6\Vt(g)
\end{equation}
with equality if and only if $g$ is Poincar\'e-Einstein.  In the case of
equality, $\Vt(g)$ agrees with the usual renormalized volume of $g$
as a Poincar\'e-Einstein metric, and \eqref{ineq} reduces to Anderson's
formula \eqref{cgbeinstein}.
\end{proposition}

\noindent
Proposition~\ref{PEineq} suggests the following variational approach to the
existence problem for Poincar\'e-Einstein metrics with prescribed
conformal infinity in dimension 4.  Given $(M,\pa M)$ and a conformal
class $[h]$ on $\pa M$, let $[\gb]$ be a conformal class extending
$[h]$ to $M$ for which $\pa M$ is umbilic and let $g$ be the associated
singular Yamabe metric.  Consider the following minimization problem:
\[
\Phi:=\inf_{[\gb]}\left( \frac14 \int_M|W|^2\,dv_g +6\Vt(g)\right) 
\]
Proposition~\ref{PEineq} implies that $\Phi \geq 8\pi^2\chi(M)$.  It also
implies that if $\Phi>8\pi^2\chi(M)$, then there does not exist a
Poincar\'e-Einstein metric
on $M$ with conformal infinity $[h]$.  If $\Phi = 8\pi^2\chi(M)$ and the
infimum is attained, then any minimizer is a Poincar\'e-Einstein metric
having conformal infinity $[h]$.  It is tempting to view this as a sort of
Dirichlet Principle for the Poincar\'e-Einstein problem.  But using it
seems to be problematic.  For starters, one must first solve for the
singular Yamabe metric to evaluate the ``energy''
$\frac14 \int_M|W|^2\,dv_g +6\Vt(g)$.  Moreover, this energy includes the
nonlocal, renormalized contribution $\Vt(g)$ which is difficult to analyze.
Finally, one must compare the infimum to $8\pi^2\chi(M)$ in order to deduce
any conclusions.

%
%

The formula \eqref{cgb} suggests that in the context
of the Chern-Gauss-Bonnet Theorem in dimension 4, it is natural to
consider metrics for which $\sigma_2(g^{-1}P)$ is constant.  This motivates
consideration here of the $\sigma_k$-Yamabe problem introduced in
\cite{V1}:  given a closed manifold $(M,\gb)$ of dimension at least three,
find a conformal metric $g = u^{-2}\gb$ satisfying  
\begin{equation} \label{skYi}
\sigma_k( g^{-1} P_g ) = const.
\end{equation}
If $k=1$ this reduces to the Yamabe problem, while if $k \geq 2$
(\ref{skYi}) is fully non-linear (as an equation for the conformal factor
$u$), and one must impose a condition to guarantee ellipticity.   To this
end, a metric $g$ is said to be {\em $k$-admissible}  
(or, if the context is clear, simply {\em admissible}) if $\sigma_j(g^{-1}
P_g) > 0$ for all $1 \leq j \leq k$.    If $\gb$ is $k$-admissible and the
constant on the right-hand side is positive, then (\ref{skYi}) is elliptic
at any solution (see Proposition 2 of \cite{V2}).  Equation (\ref{skYi}) is
also elliptic 
if $\sigma_j(- \gb^{-1} \Pb) > 0$ for all $1 \leq j \leq k$, in which case
$\gb$ is said to be {\em negative $k$-admissible}.  
For admissible metrics the existence theory for (\ref{skYi}) is well
developed; see \cite{V3}, \cite{STW2} for surveys.  In contrast to the
Yamabe problem, existence for classical solutions in the negative
admissible case is not fully understood, due 
to the lack of interior $C^2$-estimates for solutions (see Section 3.3 of
\cite{STW1} for a discussion). 

In \cite{MP}, Mazzeo-Pacard considered a singular version of the
$\sigma_k$-Yamabe problem in connection with the existence question for
Poincar\'e-Einstein metrics: given a compact manifold-with-boundary  
$(M, \partial M,\gb)$ of dimension $n+1$, construct an asymptotically  
hyperbolic metric $g = u^{-2}\gb$ solving (\ref{skYi}), where the constant 
is the value on hyperbolic space, namely $(-2)^{-k}\binom{n+1}{k}$.  
(A continuity argument shows that $g$ is automatically negative
$k$-admissible, since $g$ is asymptotically hyperbolic 
and satisfies (\ref{skYi}).)  Mazzeo-Pacard showed that   
the perturbation problem is never obstructed, so that given a solution $g =
u^{-2}\gb$ of the singular $\sigma_k$-Yamabe problem, every conformal
class sufficiently close to $[ \gb ]$ also admits a solution.  The
connection to Poincar\'e-Einstein metrics follows from the simple
observation that an asymptotically hyperbolic metric is Poincar\'e-Einstein
if and only if it solves the singular $\sigma_k$-Yamabe problem for all
$k$.  In particular, given a compactification $\rho^2 g_{+}$ of a
Poincar\'e-Einstein metric, it follows from the Mazzeo-Pacard result that
every conformal class $[\gb]$ near $[\rho^2 g_{+}]$ 
admits metrics $g_k=u_k^{-2}\gb$, $1\leq k\leq n+1$, where $g_k$ is a
solution of the singular $\sigma_k$-Yamabe problem.   

Although the result of Mazzeo-Pacard gives {\em local} existence -- i.e.,
existence of solutions in conformal classes near a given solution, the same 
issues arise as in the closed case when attempting to solve the singular
$\sigma_k$-Yamabe problem in general.  
In fact, in \cite{GSW}, Gursky-Streets-Warren gave an example of a
conformal manifold-with-boundary that does not admit a solution to the
singular $\sigma_k$-Yamabe problem for $k = n+1$ (see Proposition 6.3 in 
\cite{GSW}).  The obstruction is easy to explain: let $(M^{n+1}, \partial
M, \gb)$ be a locally conformally flat manifold-with-boundary, and suppose
$g = u^{-2} \gb$ is a solution of the singular $\sigma_k$-Yamabe problem
with $k = n+1$.  The continuity argument mentioned in the previous
paragraph shows that the Schouten tensor of $g$ is everywhere negative
definite.  Since $g$ is locally conformally flat, the curvature tensor of
$g$ is given by  
 \[
 R_{ij k \ell} = g_{ik} P_{j \ell} - g_{i \ell} P_{jk} - g_{ jk } P_{ i \ell} + g_{j \ell} P_{ik }.
 \]
If $P_g$ is negative definite, it is easy to check that $g$ has negative 
sectional curvature.  By the Cartan-Hadamard Theorem the universal cover
of $\mathring{M}^{n+1}$, the interior of $M^{n+1}$, is diffeomorphic to
$\mathbb{R}^{n+1}$.  However, it is easy to give examples where this not
the case: take $M^{n+1} = S^n \times [0,1]$, with $\gb$ the product
metric.  

Interestingly, recent work of Gonzalez-Li-Nguyen \cite{GLN} establishes the 
existence of a unique, Lipschitz continuous {\em viscosity}  
solution of the singular $\sigma_k$-Yamabe problem for domains in Euclidean
space.  Although this generalizes the classical Loewner-Nirenberg result, 
the example of Gursky-Streets-Warren illustrates that viscosity solutions
need not be classical (i.e., $C^2$) solutions. 

For our considerations here, which are based on formal asymptotics, we will 
simply assume that we have a smooth metric $\gb$ on $M$ and a defining
function $u\in C^{\infty}(\mathring{M})$ which has a polyhomogeneous
expansion at the boundary, such that $g=u^{-2}\gb$ satisfies   
\begin{equation}\label{sigmakequation}
\sigma_k(-g^{-1}P_g)=2^{-k}\binom{n+1}{k}.
\end{equation}
Henceforth, this is what we will mean by a solution of the singular
$\sigma_k$-Yamabe problem.  

The form of the expansion of $u$ at the boundary is determined by the
indicial roots of \eqref{sigmakequation},
viewed as an equation for $u$.  The indicial roots were 
calculated in \cite{MP}, but that derivation contains an error.  
As we discuss in \S\ref{calcs2}, the indicial roots are $0$ and $n+2$, and
in particular are independent of $k$.  Thus for any $k$, the expansion of
$u$ is of the form \eqref{uexpand}, where the coefficients $\cL$ and the 
$u^{(j)}$, $2\leq j\leq n+1$, are locally determined, and, of course, depend
on $k$.  Arguing exactly as in \cite{G3} (or see \cite{GoW}), it 
follows that $\Vol_g(\{r>\ep\})$ has an asymptotic  
expansion of the same form \eqref{sypvolexp}, where again each of $\cE$ and
the $c_j$'s is the integral over $\pa M$ of a local invariant of the
extrinsic geometry induced by $\gb$, which depends on $k$.  The constant
term $V=V(g,\gb)$ is the renormalized volume for $g$.  We denote by
$\cL^{\sigma_k}$, $\cE^{\sigma_k}$ the coefficients of the log terms in the
expansions \eqref{uexpand} and \eqref{sypvolexp} for a solution of the
singular $\sigma_k$-Yamabe problem.  The coefficients 
$u^{(j)}$, $\cL^{\sigma_k}$, $\cE^{\sigma_k}$ and $c_j$ depend only on
formal calculations, so are well-defined in terms of $\gb$ independently of  
existence of actual solutions $u$.   The same arguments in \cite{G3},
\cite{GoW} show that the log coefficients  
$\cL^{\sigma_k}$ and $\cE^{\sigma_k}$ are conformally invariant:  
under conformal change $\gbh=\Om^2\gb$,
one has $\widehat{\cE^{\sigma_k}}=\cE^{\sigma_k}$ and  
$\widehat{\cL^{\sigma_k}}=\big(\Om|_{\Si}\big)^{-n-1}\cL^{\sigma_k}$.  

We prove an analogue of Theorem~\ref{mainumbilic} for
solutions of the singular $\sigma_2$-Yamabe problem.  There are two major    
simplifications as compared with the case $k=1$:  the term involving the
Einstein tensor $E$ does not 
appear, and the general version of the formula and the conformal invariance
of $\Vt$ hold without the assumption of umbilicity.  The boundary term
$\cB^{\sigma_2}_{\gb}$ which enters is given by:     
\begin{equation}\label{B2}
24\cB^{\sigma_2}_{\gb}=  \pa_{\nub}\Rb+ 52 \Lo^{ij}\Wb_{0i0j}
+ 36 \tr(\Lo^3)+ \tfrac{50}{3} H|\Lo|^2 +24 \Lo^{ij}\Rb_{ij}
+ 6 H\Rb_{00} -\tfrac{10}{3} H\Rb  -\tfrac{16}{9} H^3. 
\end{equation}

\begin{theorem}\label{main2}
Let $g = u^{-2}\gb$ be a solution of the singular $\sigma_2$-Yamabe problem  
in dimension 4.  Then in the notation of Theorem~\ref{main},
\begin{equation}\label{mainformula2}
8\pi^2\chi(M)  = \frac14 \int_M|W|^2\,dv_g +6\Vt^{\sigma_2}(g),
\end{equation}
where 
\begin{equation} \label{V2}
\Vt^{\sigma_2}(g) = V(g,\gb)+\frac16 \int_{\pa M}\cB^{\sigma_2}_{\gb}
\,dv_h. 
\end{equation}
Moreover, $\Vt^{\sigma_2}(g)$ is conformally invariant, i.e. it 
is independent of the choice of compactification $\gb$.
\end{theorem}

Observe that when the boundary is umbilic with respect to $\gb$, then
$\cB^{\sigma_2}_{\gb} = \cB_{\gb}$.  An immediate consequence of this fact
is 

\begin{corollary} \label{NewtonCor} Let $(M,\partial M)$ be a compact
  four-dimensional manifold-with-boundary. 
 Suppose $g_1$ and $g_2$ are solutions of the singular $\sigma_k$-Yamabe
 problem in the same conformal class, for $k = 1$ and $k =2$ respectively.
 Let $g_1 = u_1^{-2} \gb$ and $g_2 = u_2^{-2} \gb$, where $\gb$ is a smooth
 compactification.   If $\partial M$ is umbilic with respect to $\gb$, then 
\begin{equation} \label{Newton}
V(g_2,\gb) \leq V(g_1,\gb), 
\end{equation}
and equality holds if and only if $g_1 = g_2$ is a Poincar\'e-Einstein metric.
\end{corollary}

As a by-product of our analysis of the Chern-Gauss-Bonnet
formula, we will deduce 
that the log coefficient vanishes in the volume expansion for the
$\sigma_2$-Yamabe problem:  
\begin{theorem}\label{nologs}
Let $n=3$.  Then $\cE^{\sigma_2}=0$.   
\end{theorem}  

This paper is organized as follows.  
In \S\ref{calcs} we prove Proposition~\ref{umbilicE} and
Theorem~\ref{main}.  As one would anticipate, Theorem~\ref{main} is proved
by taking a limit of the Chern-Gauss-Bonnet formula on $\{r\geq \ep\}$ as 
$\ep\rightarrow 0$.  In order to identify the boundary term $\cB_{\gb}$, we 
have to compute the expansion of the solution $u$ to one higher order than
in \cite{G3}.  In \S\ref{calcs2} we discuss the singular $\sigma_k$-Yamabe
problem and use the same sort of argument as in the case $k=1$ to prove
Theorems~\ref{main2} and \ref{nologs}.  In the process we calculate the
first few terms in the  
expansion of the solution for the singular $\sigma_2$-Yamabe problem when
$n=3$.  In \S\ref{anomalies} we discuss renormalized volume coefficients
and anomalies.  We derive a 
general result (Proposition~\ref{anomalyprop}) identifying the
infinitesimal anomaly  for the renormalized volume of an asymptotically 
hyperbolic metric in terms of the full set of its renormalized volume
coefficients.  We make explicit these 
coefficients for solutions of the singular $\sigma_k$-Yamabe problem for
$k=1$, $2$ when $n=3$.  These calculations allow us to 
give another proof by direct calculation of the conformal invariance of
$\Vt^{\sigma_k}(g)$ in these cases.  

In \S\ref{related} we consider two
other generalizations of the singular Yamabe problem.  The first is the
singular $\sigma_k(\Ric)$-problem: given $\gb$, find $g=u^{-2}\gb$
asymptotically hyperbolic so that $\sigma_k(g^{-1}\Ric_g)= const$.  It was 
shown in \cite{GSW} that this problem always has a unique solution, just
like the 
singular Yamabe problem.  The asymptotics of the solution are studied in
\cite{W} in the case of domains in Euclidean space.  We discuss a version
of the Chern-Gauss-Bonnet Theorem for 
solutions of this problem which follows by the same arguments as in the
case $k=1$ above.  Finally, we describe some results for the singular
$v_k$-Yamabe problem, where $v_k$ denotes the $k^{\text{th}}$
Poincar\'e-Einstein renormalized volume coefficient, with proofs deferred
to a future paper. 
It holds that $v_k=\sigma_k(g^{-1}P)$ when $k=2$ or $g$ is locally 
conformally flat, and it was pointed out in \cite{CF} that in several
regards, $\sigma_k(g^{-1}P)$ should be replaced by $v_k$ for $k>2$ and
general metrics.  The results are: a generalization of Theorem~\ref{nologs} 
to higher dimensions, a generalization to higher $k$ of the result of
\cite{G3}, \cite{GoW} that the Euler-Lagrange equation for the energy $\cE$
is a multiple of $\cL$, and a higher-dimensional version of the
Chern-Gauss-Bonnet Theorem for solutions of the singular $v_k$-Yamabe
problem with $2k=n+1$, generalizing a theorem of \cite{CQY} for
Poincar\'e-Einstein metrics.  These results indicate that the singular
$v_k$-Yamabe problem is perhaps the natural setting for these questions. 
However, existence of solutions of the singular   
$v_k$-Yamabe problem has not been studied for general metrics when $k>2$.  
It would be interesting to investigate the possibility of extending to
these equations the existence and uniqueness results of \cite{GLN} for
viscosity solutions.

%
%

\section{Proofs of Proposition~\ref{umbilicE} and
  Theorem~\ref{main}}\label{calcs}
We are interested in singular Yamabe metrics $g$ in dimension
$n+1$, $n\geq 2$, admitting a smooth compactification $\gb=u^2g$.
We use the normal exponential map
$\exp:[0,\de)_r\times \pa M\rightarrow M$ with respect to $\gb$
to identify a neighborhood of $\pa M$ with $[0,\de)_r\times \pa M$.  In
this identification, $\gb$ takes the form
\begin{equation}\label{geoform}
\gb=dr^2+h_r
\end{equation}
for a smooth one-parameter family of metrics $h_r$ on $\pa M$.  In
particular, $r$ is the $\gb$-distance to $\pa M$.  We denote by $h=h_0$ the
induced metric on $\pa M$.
It follows from \cite{M}, \cite{ACF}
that the defining function $u$ has an asymptotic expansion of the form
\eqref{uexpand}, where $\cL$ and the indicated $u^{(j)}$ are smooth,
locally determined functions on $\pa M$.

We first consider the asymptotics of the Einstein tensor
$E=\tf(\Ric(g)) = \Ric(g) -\frac{R}{n+1}g$. 
The conformal transformation law for Ricci applied to $g=u^{-2}\gb$ gives
\begin{equation}\label{Eform}
E=\tf\big(\Ric(\gb) +(n-1)u^{-1}\nab^2u\big) 
\end{equation}
(see \eqref{Ptransform} below).  In particular, it is clear that
$|E|_{\gb}=O(r^{-1})$.  In order to prove Proposition~\ref{umbilicE}, we
need to see that the leading term vanishes if $g$ is singular Yamabe and
$\pa M$ is umbilic.

\bigskip
\noindent
{\it Proof of Proposition~\ref{umbilicE}.}
We argue by putting $g$ into asymptotically hyperbolic normal form.
Upon choosing a representative for its conformal infinity, we can write $g$
as $g=s^{-2}\big(ds^2 + k_s\big)$ relative to an identification
$M\cong  [0,\delta)_s\times\pa M$ near $\pa M$ (typically this differs
from the identification induced by $\gb$ discussed above).
Here $k_s$ is a 1-parameter family of metrics on $\pa M$.
Because of the log terms in the expansion \eqref{uexpand}, the
compactification of $g$ with respect to a smooth defining function (for
example, $r$) is not typically $C^\infty$.  So $k_s$ and the diffeomorphism
putting $g$ into normal form typically have log terms in their expansions.
It is not hard to verify that the log terms occur far enough out that they
do not affect the subsequent argument.

It is straightforward to calculate the leading term of $\Ric(g)$ for 
$g=s^{-2}\big(ds^2 + k_s\big)$, for instance by applying the conformal
transformation law for $\Ric$ with conformal factor $s^{-2}$.  See, for
example, (2.4)-(2.6) of \cite{GH}.  The result is that $\Et:=\Ric(g) +ng$
is given by:
\begin{equation}\label{Etasymp}
\begin{split}
\Et_{ij}&=\tfrac{1}{2s}\big[(n-1)k_{ij}' +k^{pq}k'_{pq}k_{ij}\big]+O(1)\\
\Et_{i0}&=O(1)\\
\Et_{00}&=\tfrac{1}{2s}k^{pq}k'_{pq}+O(1),
\end{split}
\end{equation}
where $'=\pa_s$.  Taking the trace gives
\[
R+n(n+1)=\tr_g\Et=s^2\big(k^{ij}\Et_{ij}+\Et_{00}\big)
=nsk^{ij}k'_{ij}+O(s^2).
\]
Hence $R+n(n+1)=O(s^2)$ if and only if $k^{ij}k'_{ij}=0$ at $s=0$.
And $\pa M$ is umbilic for $g$ if and only if $\tf_k k'=0$ at $s=0$.
So if $g$ is singular Yamabe and $\pa M$ is umbilic, then
$k'|_{s=0}=0$.  In this case \eqref{Etasymp} shows that all components of
$\Et$ are $O(1)$, so also all components of $E=\tf \Et$ are $O(1)$, so
$|E|_{\gb}\in L^\infty(M)$.
\stopthm

We remark that in \eqref{Eexpand} below, we identify the leading $r^{-1}$ 
term of $E$ for a singular Yamabe metric written in the form $g=u^{-2}\gb$.    
This gives an alternate proof of Proposition~\ref{umbilicE}.

Recall that the Chern-Gauss-Bonnet formula \eqref{cgb}
implies that $\int_M \si_2(\gb^{-1}\Pb)\,dv_{\gb}$ is conformally invariant
on a compact 4-dimensional manifold without boundary.  Thus under a
conformal change $g=u^{-2}\gb$, the quantity
$\si_2(\gb^{-1}\Pb)-u^{-4}\si_2(g^{-1}P)$ must be expressible as a
divergence with respect to $\gb$.  The next lemma identifies this
divergence.
\begin{lemma} \label{s2eqn}
For $n=3$, one has
\begin{equation}\label{Pdiv}
\begin{split}
4\si_2(\gb^{-1}\Pb)=4&u^{-4}\si_2(g^{-1}P)\\
+&2\nab^\al\Big(u^{-3}|du|^2_{\gb}u_{\al} -u^{-2}(\Db u) u_{\al}
+u^{-2}u_{\al\be}u^\be +u^{-1}\Rb_{\al\be}u^{\be} -\tfrac12
u^{-1}\Rb u_{\al}\Big)
\end{split}
\end{equation}
On the right-hand side, indices are raised using $\gb$ and all covariant
derivatives (as in $u_{\al\be}$) are with respect to $\nab$.  Our sign
convention is $\Db=\nab^\al\nab_\al$.
\end{lemma}
\begin{proof}
This is a reformulation of the transformation law of
$\si_2(g^{-1}P)=\frac12(J^2-|P|^2)$ under conformal
change.  Recall that under the change $g=e^{2\om}\gb$, the Schouten tensor
transforms by
$P_{\al\be}=\Pb_{\al\be}-\om_{\al\be}+\om_{\al} \om_{\be}
-\frac12 |d\om|^2 \gb_{\al\be}$.  Setting $\om = -\log u$ gives
\begin{equation}\label{Ptransform}
P_{\al\be}=\Pb_{\al\be}+u^{-1}u_{\al\be}
-\tfrac12 u^{-2}|du|^2\gb_{\al\be},
\end{equation}
and taking the trace gives $u^{-2}J=\Jb+u^{-1}\Db u-2u^{-2}|du|^2$.
Now substitute into $4\si_2(g^{-1}P)=2(J^2-|P|^2)$ and
simplify to obtain
\[
\begin{split}
4u^{-4}\si_2(g^{-1}P)=&4\si_2(\gb^{-1}\Pb)
+6u^{-4}|du|^4-6u^{-3}\Db u|du|^2\\
&+2u^{-2}\big[(\Db u)^2 -|\nab^2 u|^2-3\Jb|du|^2\big]
-4u^{-1}\big[\Pb_{\al\be}u^{\al\be}-\Jb\Db u\big].
\end{split}
\]
Equation \eqref{Pdiv} reduces to this same relation upon expanding the
divergence.
\end{proof}

\bigskip
\noindent
{\it Proof of Theorem~\ref{main}.}
First we give the proof modulo identification of the explicit form
of $\cB_{\gb}$.  Then we calculate \eqref{B}.

First apply the Chern-Gauss-Bonnet Theorem for smooth 
manifolds-with-boundary to the metric 
$\gb$ on $\{r\geq \ep\}$ with $\ep>0$ small.  It states
\[
8\pi^2\chi(M)=\int_{r>\ep}\Big(\tfrac14 |\Wb|^2_{\gb}
+4\si_2(\gb^{-1}\Pb)\Big)dv_{\gb} + \int_{r=\ep}S\,dv_{h_\ep},
\]
where the boundary integrand $S$ can be written in the form 
\begin{equation}\label{S}
S=\Rb H-2\Rb_{00}H-2\Rb^k{}_{ikj}L^{ij}+\tfrac23 H^3 -2H|L|^2 +\tfrac43 \tr
(L^3)
\end{equation}
(see, for example, \cite{C}). 
In this formula for $S$, $L$ and $H$ refer to the second fundamental form
and mean curvature of $\{r=\ep\}$ for the metric $\gb$ with respect to the
inward pointing unit normal.  Use
$|\Wb|^2_{\gb}\,dv_{\gb} =|W|^2_g\,dv_g$,
substitute \eqref{Pdiv} and then \eqref{sig2} with $R=-12$, and integrate
the divergence by parts to obtain
\begin{equation}\label{epeq}
\begin{split}
8\pi^2\chi(M)=&\frac14 \int_{r>\ep}|W|^2_g\,dv_g
-\frac12 \int_{r>\ep}|E|^2_g\,dv_g +6\Vol_g(\{r>\ep\})  \\
-&2\int_{r=\ep}\Big(u^{-3}|du|^2_{\gb}u_0 -u^{-2}(\Db u) u_0
+u^{-2}u_{0\be}u^\be +u^{-1}\Rb_{0\be}u^\be -\tfrac12 u^{-1}\Rb
u_0\Big)dv_{h_{\ep}}\\
+&\int_{r=\ep}S\,dv_{h_\ep}.  
\end{split}
\end{equation}

As $\ep\to 0$, the first term on the right-hand side converges to
$\frac14 \int_M |W|^2_g\,dv_g$ and the last term converges to
$\int_{\pa M}S\,dv_h$.  Thus the sum of the other three terms on
the right-hand side converges as $\ep\to 0$.  However, typically each of
them diverges individually.  The expansion of $\Vol_g(\{r>\ep\})$ is given
by \eqref{sypvolexp}.   It follows from \eqref{Eform} and \eqref{uexpand}
that
\[
|E|^2_g\,dv_g = |E|^2_{\gb}\,dv_{\gb}
= \big(F_2 r^{-2} + F_3 r^{-1} +O(1)\big)\,dv_{\gb}
\]
for smooth functions $F_2$, $F_3$ on $\pa M$.
Therefore integration shows that
\begin{equation}\label{intEexpand}
\int_{r>\ep}|E|^2_g\,dv_g = a\ep^{-1} + \cF \log\frac{1}{\ep}
+\text{ fp} \int_{r>\ep}|E|^2_g\,dv_g +o(1), \qquad a,\cF\in \R,
\end{equation}
where by definition $\text{ fp} \int_{r>\ep}|E|_g^2\,dv_g$ denotes the
constant term in the expansion.  As for the boundary integral,
consideration of the form which results upon substituting \eqref{uexpand}
into each term individually shows that
\begin{equation}\label{integrand}
\begin{split}
\Big(u^{-3}|du|^2_{\gb}u_0 &-u^{-2}(\Db u) u_0
+u^{-2}u_{0\be}u^\be +u^{-1}\Rb_{0\be}u^\be -\tfrac12 u^{-1}\Rb
u_0\Big)\Big|_{r=\ep}\,dv_{h_{\ep}}\\
&=\big(B_0\ep^{-3} + B_1\ep^{-2} +B_2\ep^{-1} +B_3 +o(1)\big)\,dv_h
\end{split}
\end{equation}
for smooth locally determined functions $B_0$, $B_1$, $B_2$, $B_3$ on $\pa
M$.  The log term in $u$ does not affect the expansions to this order:  it
generates an $\ep\log \ep$ term in this expansion.

Combining the terms, we deduce first that the divergent terms must cancel:
\begin{equation}\label{canceldivs}
\begin{aligned}
2\int_{\pa M}B_0\,dv_h&=6c_0\\
2\int_{\pa M}B_1\,dv_h&=6c_1\\
2\int_{\pa M}B_2\,dv_h&=6c_2-\tfrac12 a\\
0&=6\cE-\tfrac12 \cF.
\end{aligned}
\end{equation}
Then taking the limit gives
\[
8\pi^2\chi(M)=\frac14 \int_M|W|^2_g\,dv_g
-\frac12 \text{ fp}\int_{r>\ep}|E|^2_g\,dv_g +6V 
+\int_{\pa M}(-2B_3+S)\,dv_h.
\]
This proves Theorem~\ref{main} once we carry out the calculation that
$S-2B_3=\cB_{\gb}$ modulo divergence terms.

In order to calculate $B_3$, we need to expand all the ingredients
appearing in \eqref{integrand}, namely $u$, $\gb$, and $dv_{h_{\ep}}$, to
high enough order to evaluate the constant term.  The expansions were
calculated in \cite{G3} to one order lower than required here.

Begin with $\gb=dr^2+h_r$.  Denoting $\pa_r$ by $'$, the derivatives of
$h_r$ at $r=0$ are given by:
\begin{equation}\label{derivs}
h'_{ij}=-2L_{ij},\qquad h''_{ij}=-2\Rb_{0i0j} +2L_{ik}L^k{}_j,
\qquad h'''_{ij}=-2\Rb_{0i0j,0}+8L^k{}_{(i} \Rb_{j)0k0}.
\end{equation}
The latter two equations can be derived by writing out the expressions for
the curvature components $\Rb_{0i0j}$ and $\Rb_{0i0j,0}$ in local
coordinates.  Taking the trace with respect to the metric $h=h_0$ gives
\begin{equation}\label{traces}
\tr h'=-2H,\qquad \tr h''=-2\Rb_{00} +2|L|^2
\qquad \tr h'''=-2\Rb_{00,0}+8L^{ij} \Rb_{0i0j}.
\end{equation}

Composing the expansion of $\sqrt{1+x}$ with that of the determinant shows
that for any 1-parameter family of metrics $h_r$, one has
\begin{equation} \label{dvt}
\sqrt{\frac{\det h_r}{\det{h_0}}} = 1 +D_1r +D_2r^2+D_3r^3+\cdots
\end{equation}
with
\[
\begin{aligned}
D_1&=\tfrac12 \tr h'\\
D_2&=\tfrac14\big[\tr h''-|h'|^2+\tfrac12 (\tr h')^2\big]\\
D_3&=\tfrac{1}{12}\big[\tr h'''-3\langle h',h''\rangle
+2\tr (h'^3) +\tfrac32 (\tr  h')(\tr h'')  -\tfrac32 (\tr h')|h'|^2
+\tfrac14 (\tr h')^3\big].
\end{aligned}
\]
Substituting \eqref{derivs} and \eqref{traces} gives
\begin{equation}\label{Ds}
\begin{aligned}
D_1&=-H\\
D_2&=\tfrac12\big[-\Rb_{00}-|L|^2+H^2\big]  \\
D_3&=\tfrac16\big[-\Rb_{00,0}-2L^{ij}\Rb_{0i0j}-2\tr(L^3)
  +3H\Rb_{00}+3H|L|^2-H^3\big].
\end{aligned}
\end{equation}
The above formulas hold in general dimension.

The expansion of $u$ is determined by the condition $R_g=-n(n+1)$ with
$g=u^{-2}\gb$.  Necessarily $g$ is asymptotically hyperbolic
(i.e. $|du|_{\gb}=1$ on $\pa M$), since the scalar curvature of the
conformally compact metric $u^{-2}\gb$
is asymptotic to $-n(n+1)|du|^2_{\gb}$.  Thus we write $u=r+r^2\fe$.  Now
write the equation $R_g=-n(n+1)$ in terms 
of $u$ via conformal transformation, and then write the resulting equation
in terms of $\fe$.  The result (see the derivation of (2.5) of \cite{G3})
is that $\fe$ satisfies
\begin{equation}\label{fe}
\begin{split}
(1+r\fe)&\Big[ r^2\fe_{rr}+4r\fe_r+2\fe
+\tfrac12 h^{ij}h'_{ij} (1+2r\fe+r^2\fe_r)+r^2\Delta_{h_r}\fe\Big]\\
-\frac{n+1}{2}&\Big[2(r\fe_r+2\fe)+r(r\fe_r+2\fe)^2+r^3h^{ij}\pa_i\fe\pa_j\fe\Big]\\
+\frac{1}{2n}&r(1+r\fe)^2R_{\gb}=0.
\end{split}
\end{equation}
We need to determine the Taylor expansion of $\fe$ through order 2 by
successive differentiation of \eqref{fe} at $r=0$.  The evaluation of
$\fe|_{r=0}$ and $\pa_r\fe|_{r=0}$ was given in \cite{G3}, although
there the Gauss equation was used to rewrite some expressions in terms of
intrinsic  curvature of $h$.  Here we leave everything in terms of $L$ and
curvature of $\gb$.

Setting $r=0$ in \eqref{fe} and solving for $\fe$ give
\begin{equation}\label{fe0}
\fe|_{r=0}=-\tfrac{1}{2n}H.
\end{equation}
Differentiating once, setting $r=0$, and solving for $\fe_r$ give
\[
3(n-1)\fe_r|_{r=0}=\tfrac12 \tr h''-\tfrac12 |h'|^2
+\tfrac32\fe\tr h'-2n\fe^2+\tfrac{1}{2n}\Rb.
\]
Upon substituting from \eqref{derivs}, \eqref{traces} and \eqref{fe0}, and
decomposing $|L|^2=|\Lo|^2+\frac{1}{n} H^2$, this simplifies to
\begin{equation}\label{fe1}
3(n-1)\fe_r|_{r=0}=-\Rb_{00}-|\Lo|^2+\tfrac{1}{2n}\Rb.
\end{equation}
Differentiating \eqref{fe} twice, setting $r=0$, and solving for $\fe_{rr}$
give
\[
\begin{split}
4(n-2)\fe_{rr}|_{r=0}=\tfrac12 &\tr h'''-\tfrac32 \langle h',h''\rangle +\tr
(h'^3) +3\fe\tr h''-3\fe |h'|^2+4\fe_r\tr h'\\
&+2\fe^2\tr h' +4(1-3n)\fe\fe_r
+2\Delta_h \fe +\tfrac{1}{n}\Rb_{,0}+\tfrac{2}{n}\fe\Rb.
\end{split}
\]
Henceforth we take $n=3$.  Substituting from \eqref{derivs},
\eqref{traces}, \eqref{fe0} and \eqref{fe1}, and simplfying, this can be
written
\begin{equation}\label{fe2}
\begin{split}
12\fe_{rr}|_{r=0}=-3\Rb_{00,0}+\Rb_{,0}-&\Delta_hH-6L^{ij}\Rb_{0i0j}-6\tr(L^3)\\
+&\tfrac{13}{3}H|\Lo|^2+\tfrac{13}{3}H\Rb_{00}-\tfrac59 H\Rb+\tfrac23 H^3.
\end{split}
\end{equation}

Now we can determine the $B_j$, $0\leq j\leq 3$, by expanding the
left-hand side of \eqref{integrand}.  Recalling $u=r+r^2\fe$, we write
$\fe = f_0+f_1r+f_2 r^2+o(r^2)$ with each $f_j\in C^\infty(\pa M)$, so
\[
u=r+f_0r^2+f_1r^3+f_2 r^4 +o(r^4)
=r\big(1+f_0r+f_1r^2+f_2r^3 +o(r^3)\big)
\]
and
\[
f_0=\fe|_{r=0}, \qquad f_1=\fe_r|_{r=0}, \qquad
f_2=\tfrac12 \fe_{rr}|_{r=0}
\]
are determined above.

First we evaluate the expansions of the ingredients to the
relevant orders.  Details of the verifications of these expansions are left
to the reader.
\begin{equation}\label{firstderivs}
\begin{aligned}
u_0&=1+2f_0r+3f_1r^2+4f_2r^3+o(r^3)\\
u_i&=O(r^2)
\end{aligned}
\end{equation}
so
\begin{equation}\label{du2}
|du|^2=u_0^2+O(r^4)=1+4f_0r+(6f_1+4f_0^2)r^2+(8f_2+12f_0f_1)r^3+o(r^3).
\end{equation}
The inverse powers are given by
\[
\begin{aligned}
u^{-3}&=r^{-3}\big[1-3f_0
r+(-3f_1+6f_0^2)r^2+(-3f_2+12f_0f_1-10f_0^3)r^3+o(r^3)\big]\\
u^{-2}&=r^{-2}\big[1-2f_0r+(-2f_1+3f_0^2)r^2+o(r^2)\big]\\
u^{-1}&=r^{-1}\big[1-f_0r+o(r)\big].
\end{aligned}
\]

The Christoffel symbols of $\gb$ are:
\begin{equation}\label{christofs}
\Gb_{\al\be}^0=
\begin{pmatrix}
0&0\\
0&-\tfrac12 h_{ij}'
\end{pmatrix},\qquad
\Gb_{\al\be}^k=
\begin{pmatrix}
0&\tfrac12h^{kl}h_{lj}'\\
\tfrac12 h^{kl}h_{il}'&\Gamma_{ij}^k
\end{pmatrix},
\end{equation}
where here unadorned $h$'s and $\Gamma$'s refer to
$h_r$.  Using \eqref{christofs}, one calculates for the second covariant
derivatives $u_{\al\be}=\nab_\al\nab_\be u$:
\begin{equation}\label{secondd}
\begin{aligned}
u_{00}&=\pa_{00}^2u=2f_0+6f_1r+12f_2r^2+o(r^2)\\
u_{i0}&=O(r).
\end{aligned}
\end{equation}
The tangential second covariant derivatives are
$u_{ij}=\pa_{ij}^2u-\Gb_{ij}^\al u_{\al}=\frac12
h_{ij}'u_0+\nabla_{ij}^2f_0 r^2+o(r^2)$, where $\nabla_{ij}^2$ refers to
the covariant derivatives for the metric $h_0$.  Taylor expanding $h_{ij}'$
and multiplying by $u_0$ from \eqref{firstderivs} give
\begin{equation}\label{seconddt}
u_{ij}=\tfrac12 h_{ij}'(0)+\big[\tfrac12 h_{ij}''(0)+h_{ij}'(0)f_0\big]r
+\big[\nabla_{ij}^2f_0+\tfrac14 h_{ij}'''(0)+h_{ij}''(0)f_0
+\tfrac32 h_{ij}'(0)f_1\big]r^2+o(r^2).
\end{equation}
Since
\[
h^{ij}=h^{ij}(0)-(h')^{ij}(0)r
+\tfrac12\big(-(h'')^{ij}+2(h')^{ik}(h')_k{}^j\big)(0)r^2+o(r^2),
\]
one obtains upon substituting and expanding:
\begin{equation}\label{lapl}
\begin{split}
\Db u=&u_{00}+h^{ij}u_{ij}\\
=&\big[2f_0+\tfrac12 h^{ij}h_{ij}'\big]
+\big[ 6f_1+\tfrac12 \tr h''+\tr h'f_0-\tfrac12 |h'|^2\big]r\\
&+\big[12f_2 +\Delta f_0+\tfrac14\tr h'''
+\tr h''f_0+\tfrac32\tr h'f_1 -\tfrac34\langle h',h''\rangle
-|h'|^2f_0+\tfrac12 \tr (h'^3)\big]r^2\\
&+o(r^2)
\end{split}
\end{equation}
Here all $h$ and derivatives are evaluated at $r=0$, and
$\Delta=\Delta_{h_0}$.
The remaining expansions that we need are
\[
\begin{aligned}
\Rb_{00}&=\Rb_{00}(0)+\Rb_{00,0}(0)r+o(r)\\
\Rb&=\Rb(0)+\Rb_{,0}(0)r +o(r).  
\end{aligned}
\]

Now multiply out the expansions for the terms in the left-hand side of
\eqref{integrand} term-by-term to obtain:
\begin{equation} \label{terms} 
\begin{split}
u^{-3}|du|^2u_0=r^{-3}&\Big[1+3f_0r+6f_1r^2
+\big(9f_2+3f_0f_1-2f_0^3\big)r^3 +o(r^3)\Big]\\
u^{-2}(\Db u) u_0=r^{-2}&\Big[\big(2f_0+\tfrac12\tr h'\big)
+\big(6f_1+\tfrac12 \tr h'' +\tr h'f_0-\tfrac12 |h'|^2)r\\
&+\big(12f_2+\Delta f_0 +\tfrac14 \tr h'''+\tr h''f_0
+\tfrac32\tr h'f_1-\tfrac34 \langle h',h''\rangle-|h'|^2f_0\\
&+\tfrac12 \tr (h'^3)+(f_1-f_0^2)(2f_0+\tfrac12 \tr h')\big)r^2
+o(r^2)\Big]\\
u^{-2}u_{0\be}u^{\be}=r^{-2}&\Big[2f_0+6f_1r+(12 f_2
+2f_0f_1-2f_0^3)r^2+o(r^2)\Big]\\
u^{-1}\Rb_{0\be}u^\be =
r^{-1}&\Big[\Rb_{00}+r(\Rb_{00,0}+f_0\Rb_{00})+o(r)\Big]\\
u^{-1}\Rb u_0=r^{-1}&\Big[\Rb+r(\Rb_{,0}+f_0\Rb)+o(r)\Big],
\end{split}
\end{equation} 
where coefficients on the right-hand side are again evaluated at
$r=0$.  Set 
\begin{equation}\label{Idef}
\cI=u^{-3}|du|^2_{\gb}u_0 -u^{-2}(\Db u) u_0
+u^{-2}u_{0\be}u^\be +u^{-1}\Rb_{0\be}u^\be -\tfrac12 u^{-1}\Rb
u_0.
\end{equation}
Combining terms and then substituting \eqref{derivs},
\eqref{traces}, \eqref{fe0}, \eqref{fe1}, \eqref{fe2} and setting $r=\ep$
result in
\begin{equation}\label{Iexpand}
\cI|_{r=\ep}=\ep^{-3}+\cI_1 \ep^{-2} +\cI_2\ep^{-1} +\cI_3+o(1),
\end{equation}
with
\begin{equation}\label{Is}
\begin{aligned}
\cI_1=&\tfrac12 H\\
\cI_2=&\Rb_{00}-\tfrac13 \Rb\\
24\cI_3=&9\Rb_{00,0}-3\Rb_{,0}-5\Delta H-30L^{ij}\Rb_{0i0j}-30\tr (L^3)\\
&+17H|\Lo|^2+13H\Rb_{00}-\tfrac23 H\Rb +\tfrac{26}{9}H^3.  
\end{aligned}
\end{equation}
Now \eqref{integrand} gives
\[
\begin{aligned}
B_0&=1\\
B_1&=\cI_1+D_1\\
B_2&=\cI_2+\cI_1D_1+D_2\\
B_3&=\cI_3+\cI_2D_1+\cI_1D_2+D_3.
\end{aligned}
\]
Substituting \eqref{Ds} and \eqref{Is} and simplifying give finally
\begin{equation}\label{Bs} 
\begin{aligned}
B_0&=1\\
B_1&=-\tfrac12 H\\
B_2&=\tfrac12 \Rb_{00}-\tfrac13 \Rb-\tfrac12 |L|^2\\
24 B_3&=5\Rb_{00,0}-3\Rb_{,0}-5\Delta H-38L^{ij}\Rb_{0i0j}-38\tr (L^3)\\
&\hspace{.9in} +23H|\Lo|^2-5H\Rb_{00}+\tfrac{22}{3}H\Rb
  +\tfrac{62}{9}H^3. 
\end{aligned}
\end{equation}
The definition \eqref{S} of $S$ contains the expression
$\Rb^k{}_{ikj}=\Rb_{ij}-\Rb_{0i0j}$.  Making this substitution and then
combining terms give
\begin{equation}\label{SB}
\begin{split}
12(S-2B_3)=&-5\Rb_{00,0}+3\Rb_{,0}+5\Delta H+62L^{ij}\Rb_{0i0j}+54\tr (L^3)\\
&-24L^{ij}\Rb_{ij}-47H|\Lo|^2-19H\Rb_{00}+\tfrac{14}{3}H\Rb
  -\tfrac{62}{9}H^3.
\end{split}
\end{equation}

This can be simplified by using the contracted second Bianchi identity
$\Rb_{,0}=2\Rb_{0\al,}{}^{\al}=2\Rb_{00,0}+2\Rb_{0i,}{}^i$, or equivalently
\begin{equation}\label{bian}
\Rb_{00,0}=\tfrac12 \Rb_{,0}-\Rb_{0i,}{}^i.
\end{equation}
The curvature components with exactly one zero index are given by
$\Rb_{0kjl}=\tfrac12 \big(\nabla_lh_{jk}'-\nabla_jh_{kl}'\big)$,
so $\Rb_{0j}=\frac12 h^{kl}\big(\nabla_lh_{jk}'-\nabla_jh_{kl}'\big)$.
Expanding the covariant derivative in terms of Christoffel symbols and
using \eqref{christofs}, one obtains
\[
\Rb_{0j,i}=\tfrac12 h^{kl}\nabla_i\nabla_lh_{jk}'
-\tfrac12 h^{kl}\nabla_i\nabla_jh_{kl}'
-\tfrac12 h'_i{}^k\Rb_{jk}+\tfrac12 h'_{ij}\Rb_{00}.
\]
Contracting and then substituting \eqref{derivs} give
\[
\Rb_{0i,}{}^i=-L_{ij,}{}^{ij}+\Delta H+L^{ij}\Rb_{ij}-H\Rb_{00},
\]
so substituting into \eqref{bian} shows that
\begin{equation}\label{bianfinal}
\Rb_{00,0}=\tfrac12 \Rb_{,0}+L_{ij,}{}^{ij}-\Delta H-L^{ij}\Rb_{ij}
+H\Rb_{00}.
\end{equation}
Now substitute \eqref{bianfinal} for $\Rb_{00,0}$ in \eqref{SB} and apply
\[
\begin{aligned}
\tr (L^3)&= \tr (\Lo^3) +H|\Lo|^2+\tfrac19 H^3\\
L^{ij}\Rb_{ij}&=\Lo^{ij}\Rb_{ij}+\tfrac13 H\Rb -\tfrac13 H\Rb_{00}\\
L^{ij}\Rb_{0i0j}&=\Lo^{ij}\Wb_{0i0j}+\tfrac12 \Lo^{ij}\Rb_{ij}
+\tfrac13 H\Rb_{00}
\end{aligned}
\]
to write in terms of trace-free parts.  Collecting terms gives
\[
\begin{split}
12(S-2B_3)=&\tfrac12 \Rb_{,0}-5\Lo_{ij,}{}^{ij}+\tfrac{25}{3}\Delta H
+62\Lo^{ij}\Wb_{0i0j}+54\tr (\Lo^3)\\
&+12\Lo^{ij}\Rb_{ij}+7H|\Lo|^2+3H\Rb_{00}-\tfrac53 H\Rb
  -\tfrac89 H^3.
\end{split}
\]
Upon comparison with \eqref{B}, one concludes that $S-2B_3=\cB_{\gb}$
modulo divergence terms, as claimed.
\stopthm

%
%
%

\section{The singular $\sigma_2$-Yamabe problem} \label{calcs2}

The main goal of this section is to prove Theorems~\ref{main2} and 
\ref{nologs}.  We begin by discussing the expansions 
\eqref{uexpand} and  
\eqref{sypvolexp} for solutions of the singular $\sigma_k$-Yamabe
problem on $M^{n+1}$.  
As described in the Introduction, we choose a metric $\gb\in C^\infty(M)$  
and we assume that we have a polyhomogeneous defining function $u$ 
so that \eqref{sigmakequation} holds, where $g=u^{-2}\gb$.  The
first task is to identify the relevant indicial root of the equation, which 
determines the form of the asymptotic expansion of $u$.    

Let $u_0$ be a smooth defining function for $M$ such that $g_0=u_0^{-2}\gb$
is asymptotically hyperbolic.  Thus $u_0=r+O(r^2)$, where $r$ is the
$\gb$-distance to $\pa M$.  
Consider a perturbation $u=u_0 +vr^\gamma$ for some $\ga>1$ and $v\in 
C^\infty(M)$, and set $g=u^{-2}\gb$.  The uniformly degenerate structure of 
the equation \eqref{sigmakequation} (see \cite{MP}) implies that  
\[
\sigma_k(-g^{-1}P_g)=\sigma_k(-g_0^{-1}P_{g_0})+I(\ga)vr^\gamma +o(r^\ga) 
\]
for a quadratic polynomial $I(\ga)$ called the indicial polynomial, whose 
roots are called the indical roots.  In \cite{MP}, the 
unknown was taken to be $u_{MP}$, where $u=r e^{-u_{MP}}$.
In particular, perturbing $u_{MP}$ at order
$\ga$ corresponds to perturbing $u$ at order $\ga +1$, so the indicial  
polynomial which arose in \cite{MP} was $I_{MP}(\ga)=I(\ga +1)$.  
The polynomial $I_{MP}(\ga)$ was identified near the bottom of p. 179 of
\cite{MP}:  $I_{MP}(\ga)=c_{k,n}(\ga^2-n\ga)-2k\beta_k^0$,
with roots
\[
\gamma_{\pm}=\frac{n}{2} \pm \sqrt{\frac{n^2}{4}+\frac{2k\beta_k^0}{c_{k,n}}}.
\]
Here $\beta_k^0=2^{-k}\binom{n+1}{k}$ is the constant on the right-hand 
side of \eqref{sigmakequation} and $c_{k,n}$ is the constant such that
\[
T_{k-1}(\tfrac12 I)=c_{k.n}I,
\] 
where $T_{k-1}$ is the $(k-1)$-st Newton
transform and $I$ is the $(n+1)\times (n+1)$ identity matrix.  It was
stated in \cite{MP} that $c_{k,n}=2^{1-k}\binom{n}{k}$, but in fact the
correct value is 
$c_{k,n}=2^{1-k}\binom{n}{k-1}$.  Making this correction, one obtains 
$\gamma_+=n+1$, $\gamma_-=-1$.  So for $u$ the indicial roots are $n+2$,
$0$.  The lower value of $0$ is an indicial root of the formal  
linearization of the problem, but it is not relevant for us since we 
require $u$ to vanish at $\pa M$.  So the only relevant indicial root for
$u$ is $n+2$.   The usual inductive order-by-order derivation then shows
that the expansion of a polyhomogeneous solution $u$ necessarily is of the 
form \eqref{uexpand}.  Since the volume form is given by
$dv_g=u^{-(n+1)}dv_{\gb}$, it then follows that $\Vol(\{r>\ep\})$ has an 
expansion of the form \eqref{sypvolexp} by exactly the same reasoning as in 
the case $k=1$ in \cite{G3} (or see \cite{GoW}).

\begin{proof}[Proof of Theorems~\ref{main2} and \ref{nologs}]

As in the proof of Theorem \ref{main}, we begin by applying the
Chern-Gauss-Bonnet Theorem for smooth manifolds-with-boundary for the
metric $\gb$ on $\{r\geq \ep\}$ with $\ep>0$ small: 
\[
8\pi^2\chi(M)=\int_{r>\ep}\Big(\tfrac14 |\Wb|^2_{\gb}
+4\si_2(\gb^{-1}\Pb)\Big)dv_{\gb} + \int_{r=\ep}S\,dv_{h_\ep},
\]
where $S$ is given by (\ref{S}), and $L, H$ refer to the second fundamental form
and mean curvature of $\{r=\ep\}$ for the metric $\gb$ with respect to the
inward pointing unit normal.  Using the conformal invariance of the Weyl
tensor, applying Lemma \ref{s2eqn} 
with $\sigma_2(g^{-1} P) = 3/2$, then integrating by parts give
\begin{equation}\label{epeq2}
\begin{split}
8\pi^2\chi(M)=&\frac14 \int_{r>\ep}|W|^2_g\,dv_g + 6 \Vol_g(\{r>\ep\})  \\
-&2\int_{r=\ep}\Big(u^{-3}|du|^2_{\gb}u_0 -u^{-2}(\Db u) u_0
+u^{-2}u_{0\be}u^\be +u^{-1}\Rb_{0\be}u^\be -\tfrac12 u^{-1}\Rb
u_0\Big)dv_{h_{\ep}}\\
+&\int_{r=\ep}S\,dv_{h_\ep}
\end{split}
\end{equation}

As before, as $\ep\to 0$ the first term on the right-hand side converges to
$\frac14 \int_M |W|^2_g\,dv_g$ and the last term converges to
$\int_{\pa M}S\,dv_h$.  Thus the sum of the other two terms on
the right-hand side converges as $\ep\to 0$ (though, as in the case of
singular Yamabe metrics, each of 
them diverges individually).  The expansion of $\Vol_g(\{r>\ep\})$ is given
by \eqref{sypvolexp}.   For the boundary integral, we will show that the
expansion \eqref{uexpand} implies that 
\begin{equation}\label{integrand2}
\begin{split}
\Big(u^{-3}|du|^2_{\gb}u_0 &-u^{-2}(\Db u) u_0
+u^{-2}u_{0\be}u^\be +u^{-1}\Rb_{0\be}u^\be -\tfrac12 u^{-1}\Rb
u_0\Big)\Big|_{r=\ep}\,dv_{h_{\ep}}\\
&=\big(A_0\ep^{-3} + A_1\ep^{-2} +A_2\ep^{-1} +A_3 +o(1)\big)\,dv_h
\end{split}
\end{equation}
for smooth locally determined functions $A_0$, $A_1$, $A_2$, $A_3$ on 
$\pa M$.  As before, the log term in the expansion of the solution 
$u$ does not enter into the expansion \eqref{integrand2} to this order;  it     
generates an $\ep\log \ep$ term in the expansion.  

We can immediately deduce that $\cE^{\sigma_2}=0$:  since the divergent
terms must cancel and there is no log term in \eqref{integrand2}, there
cannot be one in \eqref{sypvolexp} either.  This proves
Theorem~\ref{nologs}.      

The coefficients of the negative powers of $\ep$ must cancel as well, so it 
follows that   
\begin{equation}\label{canceldivs2}
\begin{aligned}
2\int_{\pa M}A_0\,dv_h&=6c_0,\\
2\int_{\pa M}A_1\,dv_h&=6c_1,\\
2\int_{\pa M}A_2\,dv_h&=6c_2. 
\end{aligned}
\end{equation}
Then taking the limit in \eqref{epeq2} gives
\[
8\pi^2\chi(M)=\frac14 \int_M|W|^2_g\,dv_g +6V(g,\gb)
+\int_{\pa M}(-2A_3+S)\,dv_h.
\]
This will complete the proof of Theorem~\ref{main2} once we show that
$S-2A_3=\cB_{\gb}^{\sigma_2}$ modulo divergence terms.

In order to calculate $A_3$, we need to expand all the ingredients
appearing in \eqref{integrand2}.  The expansion of $u$ is determined by the
condition 
\begin{equation} \label{S2Y}
\sigma_2( g^{-1} P_g) = \frac{3}{2}
\end{equation}
with $g=u^{-2}\gb$.  We rewrite this equation using (\ref{Pdiv}).  Expand
out the divergence on the right-hand side of \eqref{Pdiv} term-by-term
using the Leibnitz rule and multiply the whole equation by $\frac12 u^4$.
For the second and third terms inside the divergence, note that the Ricci
identity implies  
\[
\nab^{\al}\big( u_{\al\be}u^\be - (\Db u)u_\al  \big)
=  |\nab^2 u|^2 - (\Db u)^2 + \Rb_{\al\be}u^\al u^\be.   
\]
Upon collecting terms, one concludes that the equation \eqref{S2Y} can be
written as 
\begin{equation}\label{rewrite2}
\begin{split}
2u^4\sigma_2(\gb^{-1}\Pb)=&3(1-|du|^4)+3u|du|^2\Db u  +u^2|\nab^2 u|^2-u^2(\Db u)^2\\
&+\tfrac12 u^2\Rb |du|^2 
+ u^3\Rb_{\al\be}u^{\al\be}-\tfrac12 u^3 \Rb \Db u.
\end{split}
\end{equation}
As in Section \ref{calcs}, we expand $u$ as
\[
u=r+f_0r^2+f_1r^3+f_2 r^4 + O(r^5 \log r),
\]
where $f_j \in C^{\infty}(\partial M)$.  Our goal is to substitute this
into \eqref{rewrite2} and then calculate modulo $o(r^3)$.   The derivatives
of $u$ are given by \eqref{firstderivs}, \eqref{secondd}, \eqref{seconddt}, 
and $|du|^2$, $\Db u$ by \eqref{du2}, \eqref{lapl}.  Upon multiplying the
expansions and substituting \eqref{derivs} and \eqref{traces} for the
derivatives of $h$, we obtain the following expressions for each of  
the terms appearing in \eqref{rewrite2}, all modulo $o(r^3)$: 
\begin{align*}
|du|^4=&1+8f_0r+\big[12f_1+24f_0^2\big]r^2
+\big[16f_2+72f_0f_1+32f_0^3\big]r^3\\
\begin{split}
u|du|^2\Db u =& \big[2f_0-H\big]r
+\big[6f_1+10f_0^2-7Hf_0-\Rb_{00}-|L|^2\big]r^2\\
&+\big[12f_2+44f_0f_1+\D f_0 +16f_0^3-10Hf_1-18Hf_0^2\\
&\quad -7|L|^2f_0 -7\Rb_{00}f_0 -\tfrac12 \Rb_{00,0}-L^{ij}\Rb_{0i0j}
-\tr (L^3)\big]r^3 
\end{split}\\
\begin{split}
u^2|\nab^2 u|^2=& \big[4f_0^2+|L|^2\big]r^2\\
&+\big[24f_0f_1+8f_0^3 +6|L|^2f_0
+2\tr(L^3)+2L^{ij}\Rb_{0i0j}\big]r^3 
\end{split}\\
\begin{split}
u^2(\Db u)^2=& (2f_0-H)^2 r^2
+\big[24f_0f_1+8f_0^3-12Hf_1-16Hf_0^2+6H^2f_0\\
&\qquad\qquad\qquad\quad -4|L|^2f_0 -4\Rb_{00}f_0
  +2H\Rb_{00}+2H|L|^2\big]r^3   
\end{split}\\
u^2\Rb |du|^2=& \Rb r^2 + \big[\Rb_{,0}+6\Rb f_0\big]r^3\\
u^3\Rb_{\al\be}u^{\al\be}=&\big[2\Rb_{00}f_0-L^{ij}\Rb_{ij}\big]r^3\\
u^3\Rb \Db u =&(2f_0-H)\Rb r^3.
\end{align*}
Combining the terms, \eqref{rewrite2} becomes 
\[
\begin{split}
o(r^3)=&\big[-18f_0-3H\big]r 
+\big[-18f_1-42f_0^2-17Hf_0-2|L|^2-H^2-3\Rb_{00}+\tfrac12 \Rb\big]r^2\\
&+\big[-12f_2-84f_0f_1+3\D f_0 -48 f_0^3-18Hf_1
  -38Hf_0^2-11|L|^2f_0-6H^2f_0 \\
&\qquad -15\Rb_{00}f_0+2\Rb f_0-\tfrac32 \Rb_{00,0} 
+\tfrac12 \Rb_{,0} -L^{ij}\Rb_{0i0j}-\tr
(L^3)-2H\Rb_{00}-2H|L|^2\\
&\qquad -L^{ij}\Rb_{ij} +\tfrac12 H\Rb\big]r^3. 
\end{split}
\]
Setting the coefficients successively to zero and solving, it follows that  
\begin{equation} \label{fs}  
\begin{split}
f_0 =& -\tfrac{1}{6}H,  \\
18f_1 =& -2 |\Lo|^2 - 3\Rb_{00} + \tfrac{1}{2}\Rb, \\  
24f_2 =&  - 3 \Rb_{00,0}+ \Rb_{,0}-\Delta H   
-2 L^{ij}\Rb_{i0j0}-2L^{ij}\Rb_{ij} - 2\tr (L^3)\\  
&+\tfrac{7}{3} H \Rb_{00} + \tfrac{1}{9} H \Rb
+ \tfrac{5}{9} H |\Lo|^2 + \tfrac{2}{9}H^3.    
\end{split}
\end{equation}

\begin{remark}
The fact that the two solutions $u_1$ (for the singular Yamabe
problem) and $u_2$ (for the singular $\sigma_2$-Yamabe problem)  
transform the same way under conformal change implies that the first
nonzero coefficient in the expansion of $u_1-u_2$ must be a conformal
invariant up to scale.  This is confirmed upon comparing \eqref{fs} with
the $n=3$ case of \eqref{fe0} and \eqref{fe1}:  the coefficients of  
$r$ and $r^2$ agree in the expansions of $u_1$ and $u_2$, and the  
coefficients of $r^3$ differ by a multiple of $|\Lo|^2$.  
\end{remark}

As before, we define $\cI$ by \eqref{Idef}, the 
integrand of the first boundary term in (\ref{epeq2}).  Using the  
expansions in (\ref{terms}), the formulas in (\ref{fs}), and the expansion
of the 
metric in \eqref{derivs} and \eqref{traces}, we obtain \eqref{Iexpand},
this time with  
\begin{equation}\label{Is2}
\begin{aligned}
\cI_1&=\tfrac12 H\\
\cI_2&= \tfrac{1}{3}|\Lo|^2 + \Rb_{00}-\tfrac13 \Rb\\
24\cI_3&= 9 \Rb_{00,0}-3 \Rb_{,0}-5 \Delta
H + 6 L^{ij}\Rb_{0i0j} - 18 L^{ij} \Rb_{ij}  + 6 \tr (L^3)
\\ 
&\ \ - 5 H\Rb_{00} + \tfrac{16}{3} H\Rb - \tfrac{37}{3} H|\Lo|^2     
- \tfrac{10}{9} H^3. 
\end{aligned}
\end{equation}
Using the expansion of the volume form (\ref{dvt}) together with
\eqref{integrand2} gives  
\[
\begin{aligned}
A_0&=1\\
A_1&=\cI_1+D_1\\
A_2&=\cI_2+\cI_1D_1+D_2\\
A_3&=\cI_3+\cI_2D_1+\cI_1D_2+D_3.
\end{aligned}
\]
Substituting the formulas from \eqref{Ds} and \eqref{Is2} and simplifying give
\begin{equation}\label{As}
\begin{aligned}
A_0&=1\\
A_1&=-\tfrac12 H\\
A_2&=  \tfrac{1}{2} \Rb_{00}-\tfrac13 \Rb - \tfrac{1}{6} |\Lo|^2 - \tfrac{1}{6} H^2\\
24A_3&= 5 \Rb_{00,0} -3 \Rb_{,0}- 5\Delta H  
- 2 L^{ij}\Rb_{0i0j} -18 L^{ij}\Rb_{ij}- 2 \tr (L^3)  \\ 
& \ \ - 23 H\Rb_{00} +\tfrac{40}{3} H\Rb - \tfrac{43}{3} H|\Lo|^2 
  + \tfrac{26}{9} H^3.  
\end{aligned}
\end{equation}
Consequently,
\begin{equation}\label{SB2}
\begin{split}
12(S-2A_3)=&-5\Rb_{00,0}+3\Rb_{,0}+5\Delta  H 
+ 26 L^{ij}\Rb_{0i0j}- 6 L^{ij} \Rb_{ij} + 18 \tr (L^3) \\ 
&\ \ - H\Rb_{00} - \tfrac{4}{3}H\Rb - \tfrac{29}{3} H|\Lo|^2  
  -\tfrac{26}{9}H^3.
\end{split}
\end{equation}
Using the contracted second Bianchi identity \eqref{bianfinal} and
rewriting in terms of the trace-free components of the intrinsic and
extrinsic curvatures give 
\[
\begin{split}
12(S-2A_3)=&\tfrac12 \Rb_{,0}- 5 \Lo_{ij,}{}^{ij}+\tfrac{25}{3}\Delta H
+ 26 \Lo^{ij}\Wb_{0i0j}+ 18 \tr (\Lo^3)\\
&+ 12\Lo^{ij}\Rb_{ij} + \tfrac{25}{3} H|\Lo|^2+ 3 H\Rb_{00}-\tfrac{5}{3} H\Rb
  -\tfrac{8}{9} H^3.
\end{split}
\]
Upon comparison with \eqref{B2}, one concludes that $S-2A_3=\cB_{\gb}^{\sigma_2}$
modulo divergence terms, as claimed. 
\end{proof}

\section{Renormalized volume coefficients and anomalies}\label{anomalies} 

In this section we discuss renormalized volume coefficients and identify
them explicitly for the singular Yamabe problem and the singular
$\sigma_2$-Yamabe problem when $n=3$.  We prove a general result to the
effect that the infinitesimal anomaly for the renormalized volume can be
written explicitly in terms of renormalized volume coefficients.  We also 
indicate how these identifications can be used to give alternate direct
proofs of some of the results of the previous two sections.          

Let $\gb$ be a smooth metric on $M$, let $u$ be a defining function for
$\pa M$ with an asymptotic expansion of the form \eqref{uexpand}, and set
$g=u^{-2}\gb$.  If $\gb$ is written in the geodesic normal form
\eqref{geoform}, then 
\[
dv_g=u^{-n-1}dv_{\gb}=u^{-n-1}\sqrt{\frac{\det h_r}{\det h_0}}\,drdv_{h_0}. 
\]
It follows that 
\begin{equation}\label{volformexp}
dv_g=r^{-n-1}\big[
v^{(0)}+v^{(1)}r+v^{(2)}r^2+\cdots +v^{(n)}r^n+O(r^{n+1}\log
r)\big]\,drdv_{h_0}
\end{equation}
for some functions $v^{(j)}\in C^\infty(\pa M)$ called the renormalized
volume coefficients for $g$ relative to $\gb$.  The $c_j$ and $\cE$ in
\eqref{sypvolexp} are then given by  
\begin{equation}\label{cE}
c_j=\frac{1}{n-j}\int_{\pa M}v^{(j)}\,dv_{h_0},\quad 0\leq j\leq n-1,
\qquad \cE=\int_{\pa M}v^{(n)}\,dv_{h_0}.
\end{equation}
See \cite{G3}, \cite{GoW} for further discussion.

Consider first the case where $g=u^{-2}\gb$ is a solution of the singular
Yamabe problem.  In this case, in \S\ref{calcs} we wrote
$u=r(1+r\varphi)$.  {From} \eqref{volformexp}, it  
follows that the renormalized volume coefficients are determined by:
\[
(1+r\varphi)^{-n-1}\sqrt{\frac{\det h_r}{\det h_0}}=
v^{(0)}+v^{(1)}r+v^{(2)}r^2+\cdots +v^{(n)}r^n+O(r^{n+1}\log r). 
\]
Now $v^{(0)}=1$ and $v^{(1)}$, $v^{(2)}$ are derived for general $n$ in
(4.5) of \cite{G3}.  Setting $n=3$ and rewriting $v^{(2)}$ using the Gauss 
equation ((4.3) of \cite{G3}), these become 
\begin{equation}\label{vs}
\begin{aligned}
v^{(0)}&=\hphantom{-}1 \\
v^{(1)}&=-\tfrac13 H\\
v^{(2)}&=-\tfrac19 \Rb +\tfrac16\Rb_{00}-\tfrac{1}{18} H^2
+\tfrac16 |\Lo|^2.  
\end{aligned}
\end{equation}
Using the expansions of $\varphi$ and $h_r$ derived in \S\ref{calcs}, we
calculated 
\begin{equation}\label{v3}
v^{(3)}=\tfrac23 \big(\tr (\Lo^3) +\Lo^{ij}\Wb_{0i0j}\big)
+\tfrac13 L_{ij,}{}^{ij}-\tfrac16 \Delta H,\qquad n=3. 
\end{equation}
In particular, the log term coefficient in \eqref{sypvolexp} is 
given by
\begin{equation}\label{E}
\cE=\frac23 \int_{\pa M} \big(\tr (\Lo^3)
+\Lo^{ij}\Wb_{0i0j}\big)\,dv_h,\qquad n=3.
\end{equation}

Equation \eqref{sypvolexp} expresses the expansion of the volume using the 
defining function $r$ for the exhaustion.  One may choose to use other 
defining functions.  The coefficients in the expansion of
$\Vol_g(\{u>\ep\})$ were calculated in \cite{GoW}.  It is a general fact
(see e.g. \cite{GoW}) that the log term coefficient is independent of the
defining function chosen for the exhaustion.  Indeed, \eqref{E} agrees with
the coefficient of the log term in the corresponding expansion (4.15) of
\cite{GoW}.      

In the proof of Theorem~\ref{main} in \S\ref{calcs}, the identities
\eqref{canceldivs} relating the coefficients $c_j$, $\cE$ in the volume 
expansion \eqref{sypvolexp} to the integrals of the $B_j$ were deduced via
cancellation of divergences in the Chern-Gauss-Bonnet formula.  Since the
$c_j$ and $\cE$ are given in terms of the renormalized volume coefficients
by \eqref{cE}, our identifications \eqref{vs}, \eqref{v3} of these 
coefficients can be used to give an alternate proof of \eqref{canceldivs}
by direct calculation.  

One further piece of information is needed to carry out such a direct
proof.  The last two equations of \eqref{canceldivs} involve the
coefficients $a$, $\cF$, which are determined by \eqref{intEexpand}.  These
can also be calculated directly.  Starting with the conformal
transformation law \eqref{Eform} of Ricci, we derived for general $n$: 
\begin{equation}\label{Eexpand}
\begin{aligned}
E_{ij}&=-(n-1)\Lo_{ij}r^{-1}
+(n-1)\big(\tfrac{1}{2n}H\Lo_{ij}-\Wb_{0i0j}+\Lo_{ik}\Lo^{k}{}_j
+\tfrac{1}{n-1}|\Lo|^2h_{ij}\big)+O(r)\\
E_{0i}&=O(1)\\
E_{00}&=O(1).
\end{aligned}
\end{equation}
Upon setting $n=3$ and expanding, one obtains
\begin{equation}\label{E2}
|E|^2_g\,dv_g=|E|^2_{\gb}\,dv_{\gb}
=\Big(4|\Lo|^2r^{-2}+8\big(\tr(\Lo^3)+\Lo^{ij}\Wb_{0i0j}\big)r^{-1}+O(1)\Big)  
drdv_{h_0}.
\end{equation}
Thus
\begin{equation}\label{aF}
\begin{aligned}
a&=4\int_{\pa M}|\Lo|^2\,dv_h\\
\cF&=8\int_{\pa M}\big(\tr (\Lo^3)+\Lo^{ij}\Wb_{0i0j}\big)\,dv_h.
\end{aligned}
\end{equation}

Equations \eqref{canceldivs} can now be verified upon comparing the above
formulas with \eqref{Bs} for the $B_j$.  Namely, substitute \eqref{Bs}  
into the integrals in \eqref{canceldivs}, substitute \eqref{cE} for the 
$c_j$ and $\cE$, with the $v^{(j)}$ given by \eqref{vs}, \eqref{v3}, 
substitute \eqref{aF} for $a$ and $\cF$, and compare.  

Next let $g=u^{-2}\gb$ be a solution of the singular $\sigma_2$-Yamabe  
problem.  The 
renormalized volume coefficients are again defined by \eqref{volformexp}
and the coefficients in the volume expansion \eqref{sypvolexp} 
are again given by \eqref{cE}.  We calculated the $v^{(j)}$ using the 
expansions of $u$ and $h_r$ derived in \S\ref{calcs2}, analogously to the
case $k=1$.  The results are:   
\begin{equation}\label{vs2}
\begin{aligned}
v^{(0)}&=\hphantom{-}1\\
v^{(1)}&=-\tfrac13 H\\
v^{(2)}&=-\tfrac19 \Rb +\tfrac16\Rb_{00}-\tfrac{1}{18} H^2
-\tfrac{1}{18}|\Lo|^2\\
v^{(3)}&=\hphantom{-}\tfrac13 L_{ij,}{}^{ij}-\tfrac16 \Delta H.   
\end{aligned}
\end{equation}
Since $v^{(3)}$ is a divergence, this provides an alternate proof of 
Theorem~\ref{nologs} by direct calculation.  Comparing with \eqref{As}
gives an independent verification of \eqref{canceldivs2}.     

We next discuss how the anomaly of the renormalized volume can be expressed
in terms of renormalized volume coefficients.  The anomaly of the
renormalized volume $V(g,\gb)$ of a solution of the  singular
$\sigma_k$-Yamabe problem is defined to be the left-hand side of
\eqref{anomalydef}.    
An anomaly is determined by its linearization (infinitesimal anomaly):  
$\pa_t|_{t=0} V(g,e^{2t\om} \gb)$.  
For Poincar\'e-Einstein metrics with $n$ even, the infinitesimal anomaly  
is $\int_{\pa M}v^{(n)}\om\,dv_h$; see \S 3 of \cite{G1}.  Since for the 
singular $\sigma_k$-Yamabe problem the rescaling occurs on a   
manifold-with-boundary, in this case normal derivatives of $\om$ also 
appear in the infinitesimal anomaly.  The  
following result shows that in this setting, the  
infinitesimal anomaly can be identified with the full set $(v^{(0)},\cdots,
v^{(n)})$ of renormalized volume coefficients.     

\begin{proposition}\label{anomalyprop}
Let $g$ be an asymptotically hyperbolic metric with renormalized  
volume coefficients $v^{(j)}$ determined by \eqref{volformexp} and
renormalized volume $V(g,\gb)$ determined by \eqref{sypvolexp} 
relative to a compactification $\gb$.  (It is not assumed that 
$g=r^{-2}\gb$.)  Let $\om\in C^{\infty}(M)$.  Then      
\begin{equation}\label{infanomaly}
\pa_t|_{t=0} V(g,e^{2t\om} \gb)=
\int_{\pa M}\Bigg(\sum_{j=0}^n
\frac{v^{(n-j)}}{(j+1)!}\,\pa_r^j\om\Bigg)\,dv_h.   
\end{equation}
\end{proposition}
\begin{proof}
The proof follows the same outline as in \cite{GrW}, \cite{G1}, \cite{G3}.
Set $\gb_t=e^{2t\om}\gb$ and let $r_t$ denote the distance to the boundary
with respect to $\gb_t$.  Then 
$r_t=e^{\Up_t}r$ for a smooth function $\Up_t$ on $M$.  
Use the normal exponential map of $\gb$ to identify $M$ near $\pa M$
with $[0,\de)_r\times \pa M$ as above.  For fixed $x\in \pa M$ and $t>0$,
we can solve the relation 
$s=e^{\Up_t(x,r)}r$ for $r$ as a function of $s$:  $r=sb_t(x,s)$, 
where $b_t(x,s)$ is a smooth nonvanishing function.  Set $\ep_t(x,\ep) =
\ep b_t(x,\ep)$.  Then $r_t>\ep$ is equivalent to $r>\ep_t(x,\ep)$.  {From}  
\eqref{volformexp} it follows that
\[
\begin{split}
\operatorname{Vol}_g&(\{r_t>\ep\})-\operatorname{Vol}_g(\{r>\ep\}) 
=
\int_{r_t>\ep}dv_g - \int_{r >\ep}dv_g \\
=&\int_{\pa M}\int_{\ep_t}^{\ep}
\sum_{0\leq j\leq n} v^{(j)}(x) r^{-n-1+j}dr dv_{h_0} + o(1)\\
=&
\sum_{0\leq j\leq n-1}\ep^{-n+j}
\int_{\pa M}\frac{v^{(j)}(x)}{n-j}\left(b_t(x,\ep)^{-n+j}-1 
\right)dv_{h_0}\\
&  - \int_{\pa M} v^{(n)}(x)\log b_t(x,\ep)\,dv_{h_0}   
+o(1). 
\end{split}
\]
Now $V(g,\gb_t)-V(g,\gb)$ is the constant term in the expansion in $\ep$ of 
this expression, so
\begin{equation}\label{Vdiff} 
\begin{split}
V(g,\gb_t)-V(g,\gb)=&
\sum_{0\leq j\leq n-1}
\int_{\pa M}\frac{v^{(j)}(x)}{(n-j)(n-j)!}
\pa_{\ep}^{n-j}\big(b_t(x,\ep)^{-n+j}\big)\big|_{\ep=0}\,dv_{h_0} \\
& - \int_{\pa M} v^{(n)}(x)\log b_t(x,0)\,dv_{h_0}.
\end{split}
\end{equation}
Thus to evaluate $\pa_t|_{t=0} V(g,\gb_t)$, we need to identify the Taylor 
expansion in $\ep$ of $b_t(x,\ep)$ to first order in $t$.  

First consider $\Up_t$.  Since $\Up_0=0$, we have $\Up_t=O(t)$.  (Here and
in the sequel, $O(t^l)$ means $t^l$ times a smooth function of $t$ and the
other variables.)  Now $\Up_t$
is determined by the equation $|dr_t|^2_{\gb_t}=1$, which can be written 
\begin{equation}\label{Up1}
2r\pa_r\Up_t+r^2|d\Up_t|^2_{\gb}=e^{2(t\om-\Up_t)}-1.
\end{equation}
So
\begin{equation}\label{Up2}
r\pa_r\Up_t=(t\om-\Up_t)+O(t^2).
\end{equation}
Equation \eqref{Up1} implies $\Up_t=t\om$ at $r=0$.  Differentiating 
\eqref{Up2} successively at $r=0$ gives for the Taylor expansion of $\Up_t$
in $r$:  
\[
(j+1)\pa_r^j\Up_t|_{r=0}=t\,\pa_r^j \om|_{r=0} +O(t^2),\qquad j\geq 1. 
\]

Next consider $b_t$.  The defining relations of $\Up_t$, $b_t$ and $\ep_t$
show that 
\begin{equation}\label{bUp}
b_t(x,\ep)e^{\Up_t(x,\ep_t(x,\ep))}=1.
\end{equation}
Evaluating at $\ep=0$ gives $\log b_t(x,0)=-\Up_t(x,0)=-t\om(x,0)$.
Equation \eqref{bUp} implies $b_t(x,\ep)=1-\Up_t(x,\ep_t)+O(t^2)$.  
Since $\pa_\ep\ep_t=1+O(t)$, differentiating and applying the
chain rule successively give $\pa_\ep^j b_t(x,\ep)=-(\pa_r^j
\Up_t)(x,\ep_t)+O(t^2)$.  Thus 
\[
\pa_\ep^j b_t|_{\ep=0}=-\frac{t}{j+1}\pa_r^j\om|_{r=0}+O(t^2),\qquad j\geq
1. 
\]
Now $\pa_\ep(b_t^{-l})=-lb_t^{-l-1}\pa_\ep b_t=-l\pa_\ep b_t+O(t^2)$.
Differentiating further gives 
$\pa_\ep^j(b_t^{-l})=-l\pa_\ep^j b_t+O(t^2)$, so in particular   
\[
\pa_\ep^{n-j}(b_t^{-n+j})|_{\ep=0}=-(n-j)\pa_\ep^{n-j}b_t|_{\ep=0}+O(t^2)  
=t\frac{n-j}{n-j+1}\pa_r^{n-j}\om|_{r=0}+O(t^2).
\]
Substituting into \eqref{Vdiff} and applying $\pa_t|_{t=0}$ give 
\eqref{infanomaly}.
\end{proof}

One way to view Corollary~\ref{Vtinvariance} is that it asserts that in the     
umbilic case for $k=1$ and $n=3$, the functional $V(g,\gb)$ has the
property that its anomaly agrees with the anomaly of a functional given by
integration of a local curvature expression.  Likewise, Theorem~\ref{main2}  
implies the corresponding statement for the general case when $k=2$ and  
$n=3$.  This is an unusual property of a functional arising from a global 
construction.  Since the functionals arising as integrals of local
expressions are invariant under constant rescalings of $\gb$, the 
vanishing of the anomaly for $\om = constant$ is a necessary condition for
a global functional to have anomaly equal to the anomaly of the integral of
a local expression.  Via this condition, one can 
easily see, for instance, that the anomaly for the renormalized volume of    
Poincar\'e-Einstein metrics with $n$ even and the anomaly for functional
determinants of natural differential operators do not agree with the 
anomaly of the integral of a local expression.  (In these examples, the
conformal rescaling occurs on a closed manifold, rather than on a 
manifold-with-boundary as in the singular $\sigma_k$-Yamabe problem.)    

Under a rescaling $\gbh=e^{2\om}\gb$ with constant $\om$, the distance
function transforms by $\widehat{r}=e^{\om}r$, so the renormalized 
volume for a solution of the singular $\sigma_k$-Yamabe problem transforms
by $V(g,\gbh)=V(g,\gb)+\cE\om$.  Consequently, $\cE=0$ is a necessary
condition for the anomaly to equal the anomaly of the integral of a local
expression.  That $\cE^{\sigma_2}=0$ for $n=3$ is  
Theorem~\ref{nologs}, and for $k=1$, $n=3$ and $\gb$ umbilic, $\cE=0$
follows from \eqref{E}.        

Proposition~\ref{anomalyprop} can be used to give alternate proofs by
direct calculation of Corollary~\ref{Vtinvariance} and of the conformal
invariance of $\Vt^{\sigma_2}(g,\gb)$ in Theorem~\ref{main2}.  For $k=1$ it
can be used more generally to give a direct proof of the infinitesimal  
conformal invariance of \eqref{combination} 
in Theorem~\ref{main}, from which its full conformal invariance is a
consequence.
Among other things, this verifies the numerical coefficients in \eqref{B}, 
\eqref{B2} (although this computation does not test the coefficients of
$\tr(\Lo)^3$ or $\Lo^{ij}\Wb_{0i0j}$, since these are pointwise conformally
invariant).  To calculate the invariance of \eqref{combination}, one also 
needs to know the infinitesimal anomaly of ${\text {fp}} 
\int_{r>\ep}|E|_g^2\,dv_g$.  This can be  
derived in the same way as in Proposition~\ref{anomalyprop}.  The argument 
is the same except that one uses \eqref{E2} instead of \eqref{volformexp},
and it gives the same result with the $v^{(n-j)}$ in \eqref{infanomaly}
replaced by the coefficients in 
\eqref{E2}:
\begin{equation}\label{Eanomaly}
\pa_t|_{t=0}\left(\text{fp} \int_{r_t>\ep}|E|_g^2\,dv_g\right)  
=\int_{\pa M}\Big[2|\Lo|^2\pa_r\om 
+8\big(\tr(\Lo^3)+\Lo^{ij}\Wb_{0i0j}\big)\om\Big]\,dv_h.  
\end{equation}
The infinitesimal conformal change of $\cB_{\gb}$ and
$\cB_{\gb}^{\sigma_2}$ can be calculated term-by-term via usual conformal
transformation laws of local curvature quantites.  Combining with
\eqref{infanomaly}, \eqref{Eanomaly} for the infinitesimal anomalies 
and substituting 
\eqref{vs}, \eqref{v3}, \eqref{vs2} for the renormalized volume
coefficients which appear, one verifies that the infinitesimal anomalies 
vanish for \eqref{combination} for $k=1$, and for 
$\Vt^{\sigma_2}(g)=V(g,\gb)+\frac16 \int_{\pa M}\cB_{\gb}^{\sigma_2}
\,dv_h$ for $k=2$.  Consequently these expressions are conformally
invariant.  The computations are straightforward but tedious, so are
omitted.

\section{Related problems}\label{related}

In \cite{GSW}, Gursky-Streets-Warren studied the singular
$\sigma_k(\Ric)$-problem: given a Riemannian manifold-with-boundary
$(M^{n+1},\pa M,\gb)$, 
find a defining function $u$ so that $g=u^{-2}\gb$ satisfies
\begin{equation}\label{singRic}
\sigma_k(-g^{-1}\Ric_g)=n^k\binom{n+1}{k}.  
\end{equation}
The constant is again the value on hyperbolic space.  They showed   
that, just as for the singular Yamabe  
problem (and unlike the singular $\sigma_k$-Yamabe problem), there is
always a unique solution.  In \cite{W}, Wang studied the 
asymptotics of $u$ for domains in Euclidean space, and, among other things, 
showed that the indicial roots are once again $0$ and $n+2$.  These are the
indicial roots for the equation \eqref{singRic} on a general
manifold-with-boundary as well.  Thus the formal asymptotic expansion    
of $u$ again takes the form \eqref{uexpand} and the volume expansion has
the same form \eqref{sypvolexp}.  

When $n=3$, exactly the same arguments as in the case $k=1$ above can be 
used to prove a Chern-Gauss-Bonnet Theorem for solutions to the singular 
$\sigma_2(\Ric)$-problem.  The relevant identity replacing \eqref{sig2} is  
\[
4\sigma_2(g^{-1}P_g)=\tfrac19 \sigma_2(g^{-1}\Ric_g)-\tfrac49 |E_g|^2.   
\]
The same proof shows that there is a boundary term  
$\cB_{\gb}^{\sigma_2(\Ric)}$ so that if $g=u^{-2}\gb$ solves
\eqref{singRic}, then 
\[
8\pi^2\chi(M)  = \frac14 \int_M|W|^2\,dv_g
- \frac49 \text{ fp} \int_{r>\ep}|E|^2\,dv_g +6V(g,\gb)
+\int_{\pa M}\cB_{\gb}^{\sigma_2(\Ric)} \,dv_h. 
\]
Since the constant multiplying $\int |E|^2$ is again negative, the
subsequent conclusions about the specialization to the umbilic case 
also hold, including the variational characterization of
Poincar\'e-Einstein metrics based on Proposition~\ref{PEineq}.

We conclude with a discussion of another variant of the singular 
$\sigma_k$-Yamabe problem, this one motivated by Theorem~\ref{nologs}.
Recall that in \cite{G3} and \cite{GoW} it was proved that for the 
singular Yamabe problem, the variation of $\cE$ is a nonzero multiple of
$\cL$, where $\cE$ is viewed as a conformally invariant energy of the
varying hypersurface $\Sigma=\pa M$ in the conformal manifold  
$(M,[\gb])$.  Otherwise stated, $\cL=0$ is the Euler-Lagrange equation for
the functional $\cE$, thought of as a function of $\Si$.  If this
variational relation were true also for $k=2$ and   
$n=3$, we could conclude from Theorem~\ref{nologs} that $\cL^{\sigma_2}=0$,  
i.e. the expansion of the solution $u$ has no log terms.       
This also raises the question of whether this variational
relation between $\cE$ and $\cL$ holds more generally.

A crucial ingredient in our analysis of the Chern-Gauss-Bonnet formula
was the divergence identity \eqref{Pdiv}.  Divergence structure also
plays an important role in the proof for the singular Yamabe problem of the  
variational relation between $\cE$ and $\cL$.  It has been known for some
time that such divergence structure is lacking in general for
$\sigma_k(g^{-1}P)$  for $k>2$, and Branson-Gover showed in \cite{BG}
that the equation $\sigma_k(g^{-1}P)=const$ for $k>2$ is an Euler-Lagrange
equation if and only if $g$ is locally conformally flat.  
However, Chang-Fang realized in \cite{CF} that a modification of  
$\sigma_k(g^{-1}P)$, which we denote $v_k(g)$, has a variational/divergence 
structure, leading to the conclusion that at least for some purposes, 
$v_k(g)$ is the correct replacement for $\sigma_k(g^{-1}P)$ for $k>2$.
If $k=1$ or $2$, or if $3\leq k\leq n$ and $g$ is locally conformally flat,
then $v_k(g)=\sigma_k(g^{-1}P)$.  These observations motivate us to   
consider the ``singular $v_k$-Yamabe problem'':  given 
$(M^{n+1},\pa M,\gb)$, find a defining function $u$ so that $g=u^{-2}\gb$
satisfies
\begin{equation}\label{vkeq}
v_k(g)=(-2)^{-k}\binom{n+1}{k}.
\end{equation}

We briefly recall the definition of $v_k(g)$ and refer to \cite{CF},
\cite{G2} for details and elaboration.  Consider the formal asymptotics of
Poincar\'e-Einstein metrics:  if $g$ is a metric on a manifold $M^n$, one
searches for a metric $g_+=r^{-2}(dr^2 +g_r)$ which satisfies 
$\Ric(g_+)=-ng_+$ to high order, where $g_r$ is a one-parameter family of
metrics on $M$ with $g_0=g$.  For $n$ even, this determines the Taylor
expansion of $g_r$ to order $n$, and for $n$ odd, this, together with the 
condition that $g_r$ be even in $r$, determines the Taylor expansion of
$g_r$ to infinite order.  Then one considers the expansion of the volume
form $dv_{g_r}=\big(1+v^{(2)}r^2+\ldots\big)dv_g$.
The curvature quantities $v_k$ are a multiple of the renormalized 
volume coefficients appearing in this expansion:  $v_k=(-2)^k v^{(2k)}$.
So $v_k$ is defined for all $k\geq 0$ for $n$ odd, but only for $k\leq n/2$
for $n$ even for general metrics.  However, if $g$ is Einstein or locally
conformally flat, it is possible to continue the expansion of
$g_r$ to infinite order, uniquely upon imposing an appropriate auxiliary
condition, so in these cases $v_k(g)$ is defined for all $k$ also for $n$
even.  It turns out that $v_k(g)=\sigma_k(g^{-1}P_g)$ if $k=1$ or $2$, and 
also for $3\leq k\leq n$ if $g$ is locally conformally flat or Einstein.  

For the singular $v_k$-Yamabe problem, we begin with $(M^{n+1},\pa M,\gb)$,
and $g=u^{-2}\gb$ is supposed 
to satisfy \eqref{vkeq}.  So we must replace $n$ by $n+1$ in the above
definition of $v_k(g)$.  We always assume $k\leq n+1$, and also require
$2k\leq n+1$ if $n$ is odd and $g$ is not locally conformally flat.  The
indicial roots of the equation \eqref{vkeq},  
viewed as an equation for $u$, are again $0$ and $n+2$.  So once again, the
formal expansion of $u$ has the form \eqref{uexpand} and the volume
expansion has the form \eqref{sypvolexp}.  Just as for the problems
discussed above, the coefficients $\cE^{v_k}$ and $\cL^{v_k}$ are
determined by formal calculations alone, so are well-defined independently
of existence theory for the equation.  And both of them satisfy the same
conformal invariance relations as before:  under conformal change
$\gbh=\Om^2\gb$, one has $\widehat{\cE^{v_k}}=\cE^{v_k}$ and  
$\widehat{\cL^{v_k}}=\big(\Om|_{\Si}\big)^{-n-1}\cL^{v_k}$.  

The first result is a generalization of Theorem~\ref{nologs} to the
singular $v_k$-Yamabe problem:  
\begin{theorem}\label{Ezero}
If $n\geq 3$ is odd and $2k=n+1$, then $\cE^{v_k}=0$.      
\end{theorem}

The second result is a generalization to $k>1$ of the variational 
relationship between $\cE$ and $\cL$.  To formulate this result, note that
$\cE^{v_k}$ and $\cL^{v_k}$ are determined just by the local geometry of
$\pa M$ in $(M,\gb)$, so they can be defined for a general hypersurface 
$\Si$ with chosen normal direction in a Riemannian manifold $(M^{n+1},\gb)$
(it must be assumed for $\cE^{v_k}$ that $\Si$ is compact to carry out the
integration).  Suppose that  
$F_t:\Si\rightarrow M$, $0\leq t\leq \delta$, is a variation of $\Si$,  
i.e. a smoothly varying one-parameter family of embeddings with
$F_0=\text{Id}$.  Set $\Si_t=F_t(\Si)$ and denote by $\cE^{v_k}_t$ the 
corresponding quantity for $\Si_t$.  Write 
$\dot{F}=\pa_tF|_{t=0}\in \Gamma(TM|_{\Si})$ and
$\dot{\cE^{v_k}}=\pa_t\cE^{v_k}_t|_{t=0}$.  Let $\nub$ denote the inward
pointing  $\gb$-unit normal to $\Si$ in $M$.  
\begin{theorem}\label{var}
Suppose $n\geq 2$ and $1\leq k\leq n+1$.  Suppose also that $2k\leq n+1$ if 
$n$ is odd and $\gb$ is not locally conformally flat.  Then    
\[
\dot{\cE^{v_k}}
=(n+2)(n-2k+1)\int_\Si\;
\langle \dot{F},\nub\rangle_{\gb}\,\cL^{v_k}\,dv_\Si.      
\]
\end{theorem}

\noindent
Thus the coefficient relating $\dot{\cE^{v_k}}$ and $\cL^{v_k}$ vanishes
when $2k=n+1$, and in this case one can make no conclusions about
$\cL^{v_k}$ from the fact that $\cE^{v_k}=0$.   In particular, when $k=2$,
$n=3$, there is no conclusion about $\cL^{\si_2}$ from the fact that 
$\cE^{\si_2}=0$.  

Finally, we state a version of the Chern-Gauss-Bonnet Theorem in higher  
dimensions for solutions of the singular $v_k$-Yamabe problem, $2k=n+1$.
This is 
motivated by Theorem~\ref{main2} above and by Theorem 3.3 in \cite{CQY},
which generalizes Anderson's formula to Poincar\'e-Einstein metrics in
higher even dimensions.  For this  
result, we assume that our solution of the $v_k$-Yamabe problem is 
smooth in $\mathring{M}$ with a polyhomogeneous expansion at $\pa M$.  Its 
renormalized volume $V(g,\gb)$ is defined as usual by \eqref{sypvolexp}.   
\begin{theorem}\label{higherCGB}
Let $n\geq 3$ be odd and $2k=n+1$.  There is a scalar pointwise conformal  
invariant $J$ of weight $-(n+1)$ and a boundary term $\cB_{\gb}^{v_k}$ so
that if $g=u^{-2}\gb$ is a solution of the singular $v_k$-Yamabe problem
which is smooth in $\mathring{M}$ and polyhomogeneous at $\pa M$, then  
\begin{equation}\label{CGBformula}
c_n\chi(M)=\int_M J_g\,dv_g +\Vt(g),\qquad
c_n=\frac{(-1)^{\frac{n+1}{2}}\pi^{\frac{n+2}{2}}}{\Gamma(\frac{n+2}{2})},
\end{equation}
where
\[
\Vt(g)=V(g,\gb)+\int_{\pa  M}\cB_{\gb}^{v_k}\,dv_{\gb}
\]
is independent of the choice of compactification $\gb$.   
\end{theorem}
\noindent
Note that the conformal invariance of $J$ implies that 
$J_gdv_g=J_{\gb}dv_{\gb}$, so that $\int_M J_gdv_g$ converges.   
Proofs of Theorems~\ref{Ezero}, \ref{var}, and \ref{higherCGB} will be
given elsewhere.     

We remark that in dimension 2, constant Gauss curvature metrics play the
role of both Einstein metrics and metrics of constant sectional curvature.
So for $n=1$, every singular Yamabe metric 
should be regarded as Poincar\'e-Einstein.  The quantities discussed here
for singular Yamabe metrics have the same properties when $n=1$ as for
Poincar\'e-Einstein metrics with $n\geq 3$ odd. Namely, 
$\cE$ and $\cL$ both vanish (see \cite{G3}), and the renormalized volume
defined using a geodesic defining function for $g$ is conformally
invariant, i.e. independent of the geodesic defining function.  In this
case, the analogue of Theorem~\ref{higherCGB} is the result (\cite{Ep})
that the renormalized volume defined using a geodesic defining function
equals 
$-2\pi \chi(M)$.  One can also consider the renormalized volume defined
using an arbitrary defining function (corresponding to choosing an
arbitrary compactification $\gb$), in which case the analogue of
\eqref{CGBformula} takes the form
\[
-2\pi \chi(M) = V(g,\gb) +\frac12 \int_{\pa M}H_{\gb}\,ds_{\gb}.   
\]

\end{document}